\documentclass{elsart}

\usepackage[ruled,vlined]{algorithm2e}
\usepackage{amssymb}
\usepackage{graphicx}
\usepackage[margin=1.1in]{geometry}

\usepackage{graphicx}
\usepackage{color}\usepackage{stmaryrd}
\usepackage{amsmath}\usepackage{amsfonts}\usepackage{amssymb}\usepackage{bm}
 \usepackage{amsthm}
\newtheorem{theorem}{Theorem}[section]
\newtheorem{lem}{Lemma}[section]
\newtheorem{corollary}{Corollary}[section]
\newtheorem{proposition}{Proposition}[section]
\newtheorem{remark}{Remark}[section]
\numberwithin{equation}{section}
\numberwithin{theorem}{section}
\numberwithin{remark}{section}

\def\b#1{{\bf#1}}
\def\x{{\bf X}} \def\l2{\mathbf{L}^2(\Omega)}
\def\c0{\mathbf{H}_0(\mathrm{curl},\Omega) }

\def\d0{\mathbf{H}(\mathrm{div0},\Omega) }
\def\rt{\rightarrow} \def\wt{\widetilde} 
\def\1{C_1} \def\2{C_2} \def\3{C_3} \def\4{C_4} \def\5{C_5} \def\6{C_6}
\def\7{C_7} \def\8{C_8} \def\9{C_9} \def\0{C_0}
\newcommand{\n}[1]{{\left\vert\kern-0.25ex\left\vert\kern-0.25ex\left\vert #1
    \right\vert\kern-0.25ex\right\vert\kern-0.25ex\right\vert}}

\begin{document}

\begin{frontmatter}

\title{An $hp$-version interior penalty discontinuous Galerkin method   for the  quad-curl  eigenvalue problem\thanks{Project supported by
the National Natural Science Foundation of China(Grants No. 12001130, 1871092, NSAF 193040.).}}

\author[a1,a2]{Jiayu Han}\ead{hanjiayu@csrc.ac.cn}
\author[a1,a3]{Zhimin Zhang}\ead{zmzhang@csrc.ac.cn}
\address[a1]{Beijing Computational Science Research Center, Beijing, 100193, China}
\address[a2]{School of Mathematical Sciences, Guizhou Normal University, 550025, China}
\address[a3]{Department of Mathematics, Wayne State University, Detroit, MI 48202, USA}

\begin{abstract}  {An $hp$-version interior penalty discontinuous Galerkin (IPDG) method under nonconforming meshes is proposed to solve the quad-curl eigenvalue problem. We prove well-posedness of the numerical scheme for the quad-curl equation and then derive an error estimate in a mesh-dependent norm, which is optimal with respect to $h$ but has different p-version error bounds under conforming and nonconforming tetrahedron meshes. The $hp$-version discrete compactness of the DG space is established for the convergence proof. The performance of the method is demonstrated by numerical experiments using conforming/nonconforming meshes and $h$-version/$p$-version refinement. The optimal $h$-version convergence rate and the exponential $p$-version convergence rate are observed.}
\end{abstract}

\end{frontmatter}

\section{Introduction}
The quad-curl eigenvalue problem {is  important} in inverse electromagnetic scattering theory of inhomogeneous media \cite{cakoni,monk1,sun1} and magnetohydrodynamics equations \cite{zheng}. As a classic electromagnetic model, the Maxwell eigenvalue problem is a  hot topic in the field of numerical mathematics and  computational electromagnetism (see, e.g., \cite{ainsworth,boffi2,boffi1,buffa1,buffa2,brenner,ciarlet1,hiptmair1,kikuchi,monk,monk2,reddy,russo}).
In recent years,  the numerical treatment of  the quad-curl equation and its {associated} eigenvalue problem {has} attracted the attention of scientific community. {Some early works include a nonconforming element method by Zheng et al. \cite{zheng} and  an $h$-version IPDG method by Hong et al. \cite{hong}, respectively, for the quad-curl equation}. Chen et al. \cite{chen} and Sun et al. \cite{sun2} further
established its new a-priori error estimates and multigrid method, respectively.
Sun et al. \cite{sun1} {also proposed an} $h$-version weak Galerkin method for the quad-curl equation.
  Sun \cite{sun} and Zhang \cite{zhang}   studied its mixed element methods, respectively. Brenner et al. \cite{brenner1} proposed its  Lagrange finite element methods  on planar domains. 
 More recently  Zhang and Hu et al. \cite{zhang1,hu,hu1}  proposed {several families of} $\b H(\mathrm{curl}^2)$-conforming finite {elements in both two and three dimensions. A-priori and a-posteriori error} estimates for  the quad-curl eigenvalue problem were further {developed}  in \cite{wang}. The  $\b H(\mathrm{curl}^2)$-conforming virtual element methods   \cite{zhao} and  the decoupled finite element method  \cite{cao}  were also proposed.

The $hp$   finite element methods are popular in scientific computing due to {its flexibility} and
high accuracy. DG methods  provide general framework for $hp$-adaptivity as they employ the discontinuous finite element spaces, giving great flexibility in
the design of meshes and polynomial bases.  For the overview of the historical development of $hp$-version DG methods, we
refer {to, e.g., articles \cite{arnold,baker},} monographs \cite{cockburn,pietro} and  the references therein.
 For the second-order elliptic problem and the biharmonic problem, a considerable number of works on $hp$-version IPDG methods were done, cf.
\cite{houston2002,mozolevski1,mozolevski2,suli,georgoulis1,georgoulis2,feng,karakashian,pan}. 

   However, to the best of our knowledge, the research work on the $hp$-version  DG methods for the quad-curl equation and its eigenvalue problem cannot be found in the existing literatures. This paper  aims to fill this gap. We propose an $hp$-version  IPDG method with two interior penalty parameters to solve the quad-curl equation and associated eigenvalue problem.
{The} $hp$-version  error estimates for {the} DG solution of  the quad-curl equation (without  div-free constraint) are established under the  assumption that the exact solution  $\bm w\in \b H^{3}(\Omega)$ and $\mathrm{curl}^3 \bm w\in \b H^{1/2+\delta}(\Omega)$.
  To bound the error  of the IPDG method for the  eigenvalue problem,
 {we analyze} the $hp$-version   DG discretization of  the quad-curl  equation with div-free constraint.
 {A} discrete poinc\'{a}re inequality in {the discrete div-free space is established} to guarantee the well-posedness of the DG discretization scheme.   Then {$hp$-version discrete compactness on the discrete div-free space} is   established to prove the unform convergence of discrete solution operators. {Finally, we use} the well-known Babu\u{s}ka-Osborn theory \cite{babuska}
  {to prove the convergence of the IPDG method for} the quad-curl eigenvalue problem. The  a priori error bound of the DG eigenvalues is
optimal in  $h$ on both conforming and nonconforming {meshes}, and suboptimal in $p$ by 3 order on nonconforming simplex mesh
 {and  by} 2 order on conforming simplex mesh. 

This paper is structured as follows.  In Section 2, an $hp$-version IPDG scheme is  given for the quad-curl equation without $div$-free condition. In Section
3, we will discuss the stability of the IPDG scheme and its a-priori error estimates in DG norm under  conforming/nonconforming mesh. An IPDG scheme for the quad-curl eigenvalue problem is proposed in Section 4. The discrete poinc\'{a}re inequality and discrete compactness of discrete $div$-free space are established. The error bound for IPDG eigenvalues and the error estimates of eigenfunctions in low norms will follow.
In the end of this paper, we present several numerical examples to validate the efficiency of our methods under both $h$-refinement and $p$-refinement modes.

Throughout this paper,   we use
the symbol $a \lesssim b$ and $a \gtrsim b$   to mean that $a \le Cb$ and $a \ge Cb$ respectivley, where $C$ denotes a positive constant independent of   mesh parameters and polynomial degrees
and may not be the same   in different places.

\section{An $hp$-version IPDG method for the quad-curl  problem}
Consider the quad-curl problem
\begin{align}\label{s1}
 \mathrm{curl}^4\bm w +\bm w&=\bm f~~~in~\Omega,\\
 \mathrm{curl} \bm w\times \bm n&=0 ~~~on~\partial \Omega,\label{s3}\\
 \bm w\times \bm n&=0 ~~~on~\partial \Omega,\label{s4}
\end{align}
where $\Omega$ is a bounded simply-connected Lipschitz polyhedron domain in $\mathbb{R}^d$($d=2,3$), and $\b n$ is the unit outward normal to $\partial\Omega$.

{We adopt the following function space
}
\begin{eqnarray*}
 &&\mathbf{H}(\mathrm{curl}^s,\Omega):=\{\bm v\in \mathbf L^2(\Omega): \mathrm{curl}^j \bm v \in\mathbf L^2(\Omega),1\le j\le s\}
\end{eqnarray*}
equipped with the norm $\|\cdot\|_{s,\mathrm{curl}}$
and
\begin{eqnarray*}
 &&\mathbf{H}_0(\mathrm{curl}^2,\Omega):=\{\bm v\in \mathbf L^2(\Omega): \mathrm{curl}^j \bm v \in\mathbf L^2(\Omega),\mathrm{curl}^{j-1} \bm v\times \bm n|_{\partial \Omega}=0 ,1\le j\le 2\}
\end{eqnarray*}

The weak form of (\ref{s1})-(\ref{s4}) is to find $\bm w\in \mathbf{H}_0(\mathrm{curl}^2,\Omega) $
  such that
\begin{eqnarray}\label{2.5}
a(\bm w,\mathbf{v})= ( \bm f,\mathbf{v}),~~\forall \mathbf{v}\in \mathbf{H}_0(\mathrm{curl}^2,\Omega),
\end{eqnarray}
where $$a(\bm w,\mathbf{v})=(\mathrm{curl}^2\bm w,\mathrm{curl}^2\bm{v})+(\bm w,\bm{v}).$$


We consider the
shape regular and affine meshes $\mathcal T_{h} = \{K\}$ that partition the domain $\Omega$  into {tetrahedra in $\mathbb R^3$ (or triangles in $\mathbb R^2$)}, and
 introduce the finite element space
\begin{align}
\b V_{h  } &=\{\bm v_{h  }|_ K \in P(K),   K\in\mathcal T_{h},p_K \ge 2\}.
\end{align}
where $P(K)=(P_{p_K})^d$ with the polynomial space of total $\text{degree}\le p_K$. Let $h=\max_{K\in\mathcal T_h}h_K$ and $p=\min_{K\in\mathcal T_h}p_K$. 
Let $\mathcal{E}^0_h$ and $\mathcal{E}^{\partial}_h$ be respectively the set of  internal faces and the set of boundary faces of partition $\mathcal T_{h}$, $\mathcal{E}_h:=\mathcal{E}^0_h\cup\mathcal{E}^{\partial}_h$, $f \in \mathcal{E}^0_h$ be the interface of two adjacent
elements $ K^\pm$, and $\bm n^\pm$ be the unit outward normal vector of the face $f$ associated with $ K^\pm$.  {We use $h_f:=\max(h_{K^+},h_{K^-})$ and $p_f:=\min(p_{K^+},p_{K^-})$ to represent the maximum  diameter of the circumcircle and the maximum   polynomial degree of the elements sharing the face $f$, respectively}. {We denote by $\bm v^\pm = (\bm v|_{K^\pm} )|_f$ and introduce three notations as follows:
\begin{align*}
\llbracket  \bm v\rrbracket  &=  \bm v^+\times \bm n^+ +  \bm v^-\times \bm n^-,~~\llbracket \mathrm{curl}\bm v\rrbracket  &= \mathrm{curl}\bm v^+\times \bm n^+ + \mathrm{curl}\bm v^-\times \bm n^-,~~
\{\mathrm{curl}^2\bm v\} &=
(\mathrm{curl}^2 \bm v^+ + \mathrm{curl}^2 \bm v^-)/2.
\end{align*}
If $f \in \mathcal{E}^{\partial}_h$,  we define $\llbracket \bm  v\rrbracket $, $\llbracket \mathrm{curl}\bm  v\rrbracket $ and $\{\mathrm{curl}^2\bm  v\}$ on $f$,
respectively, as follows:
\begin{align*}
\llbracket \bm v\rrbracket  = \bm v\times \bm n,~~
\llbracket \mathrm{curl}\bm v\rrbracket  = \mathrm{curl}\bm v\times \bm n,~~
\{\mathrm{curl}^2\bm v\} =
\mathrm{curl}^2 \bm v.
\end{align*}}
Pick up any $\bm v_{h  }$ in $\b V_{h  }$. For any $ K\in\mathcal T_{h}$, by using Green's formula we have
\begin{align}\label{1}
    \int_ K\bm f\cdot \bm v_{h  }&=\int_{\partial K}\mathrm{curl}^3\bm w\cdot\bm v_{h  }\times \bm nds
    +\int_ K \mathrm{curl}^2 \bm w\cdot \mathrm{curl}^2 \bm v_{h  }dx+\int_ K \bm w\cdot \bm v_{h  }dx\nonumber\\
    &~~~+\int_{\partial K}\mathrm{curl}^2\bm w\cdot (\mathrm{curl}\bm v_{h  }\times \bm n)ds.
\end{align}
This infers that
\begin{align}\label{1}
    \int_\Omega\bm f\cdot \bm v_{h  }dx&=\int_\Omega \mathrm{curl}_h^2 \bm w\cdot \mathrm{curl}_h^2 \bm v_{h  }dx+\int_{f\in \mathcal{E}_h}\mathrm{curl}^3\bm w\cdot \llbracket \bm v_{h  }\rrbracket ds
    +\int_\Omega \bm w\cdot \bm v_{h  }dx\nonumber\\
    &~~~+\int_{f\in \mathcal{E}_h}\mathrm{curl}^2\bm w\cdot \llbracket \mathrm{curl}\bm v_{h  }\rrbracket ds
\end{align}
whose right-hand side  can be rewritten as  the bilinear form
\begin{align}\label{1}
    a_h(\bm w, \bm v_{h  })&=\int_\Omega \mathrm{curl}_h^2 \bm w\cdot \mathrm{curl}_h^2 \bm v_{h  }dx+\int_\Omega \bm w\cdot \bm v_{h  }dx
    \nonumber\\
    &~~~+\int_{f\in \mathcal{E}_h}\{\mathrm{curl}^3\bm w\}\cdot \llbracket \bm v_{h  }\rrbracket ds+\int_{f\in \mathcal{E}_h}\{\mathrm{curl}^2\bm w\}\cdot \llbracket \mathrm{curl}\bm v_{h  }\rrbracket ds\nonumber\\
    &~~~+\int_{f\in \mathcal{E}_h}\{\mathrm{curl}^3\bm v_{h  }\}\cdot \llbracket \bm w\rrbracket ds+\int_{f\in \mathcal{E}_h}\{\mathrm{curl}^2\bm v_{h  }\}\cdot \llbracket \mathrm{curl}\bm w\rrbracket ds\nonumber\\
    &~~~+\int_{f\in \mathcal{E}_h}\frac{\eta_1 p^2_f}{h_f}\llbracket \mathrm{curl}\bm v_{h  }\rrbracket \cdot \llbracket \mathrm{curl} \bm w\rrbracket ds+\int_{f\in \mathcal{E}_h}\frac{\eta_2p^6_f}{h_f^3}\llbracket \bm v_{h  }\rrbracket \cdot \llbracket \bm w\rrbracket ds
\end{align}
where we have used the jump $\llbracket \mathrm{curl}\bm w\rrbracket $ and $\llbracket \bm w\rrbracket $ vanishes across   faces $f$ and $\eta_1$ and $\eta_2$ are two positive constant to be determined.
Finally we reach at the following relation
\begin{align}\label{2.9a}
    a_h(\bm w, \bm v_{h  })&=(\bm f,\bm v_{h  }),\quad\forall\bm v_{h  }\in \b V_{h  }.
\end{align}
Hence the IPDG discretization of (2.5) is to find $\bm w_{h  }\in \b V_{h  }$ such that
\begin{align}\label{2.10aa}
    a_h(\bm w_{h  }, \bm v_{h  })&=(\bm f,\bm v_{h  }),\quad\forall\bm v_{h  }\in \b V_{h  }.
\end{align}
The above two equalities give the following Galerkin orthogonality
\begin{align}\label{2.11a}
    a_h(\bm w_{h  }-\bm w, \bm v_{h  })&=0,\quad\forall\bm v_{h  }\in \b V_{h  }.
\end{align}
\begin{lem}[Lemma 4.5 in \cite{babuska1}]
There are linear continuous operators $\Pi_p^K:\b H^{s}(K)\to P(K)$ such that
for $\bm v\in \b H^{s}( K)$
\begin{align}
 \|\bm v-\Pi_p^K\bm v\|_{q, K}&\lesssim h_K^{\min(p_K+1,s)-q}p_K^{q-s}\|\bm v\|_{s, K},~0\le q\le s,
\end{align}
\end{lem}
Let $\bm v\in\b V_{h  }$ be equipped with the following {norm and semi-norm}
\begin{align}\label{0}
    ||\bm v||_h^2&:=\sum_{K\in\mathcal T_{h}} \left(\|\mathrm{curl}^2 \bm v\|_{0, K}^2+\|\bm v\|_{0, K}^2\right)+\sum_{\bm f\in \mathcal{E}_h}\left(
h_f^{-3}p_f^{6}\|\llbracket \bm v\rrbracket \|_{0,f}^2 \right)\nonumber
    \\&~~~+\sum_{\bm f\in \mathcal{E}_h}\left(h_f^3p_f^{-6}\|\{\mathrm{curl}^3\bm v\}\|_{0,f}^2+h_fp_f^{-2}\|\{\mathrm{curl}^2\bm v\}\|_{0,f}^2+
    h_f^{-1}p_f^2\|\llbracket \mathrm{curl}\bm v\rrbracket \|_{0,f}^2\right),\\\label{1}
    |\bm v|_h^2&:=\|\bm v\|_h^2-\|\bm v\|^2.
\end{align}
Next we shall discuss the
well-posedness of the discrete variational problem.
It is obvious that the bilinear form $a_h(\cdot, \cdot)$ is bounded with respect to the norm $\|\cdot\|_h$ in $\b V_{h  }+\b H_0(\mathrm{curl}^2,\Omega)\cap \{\bm v:\sum_{j=0}^{3}\|\mathrm{curl}^j \bm v\|_{1/2+\sigma,\Omega}<\infty\}$ with $\sigma>0$.

\section{Error estimate in DG-norm}
\begin{lem}[Lemma 6.1 in \cite{herbert}]
For any polynomial $\bm v\in \b P(K)$, there holds
\begin{align}\label{1}
    \|\bm v\|_{0,\partial K}\lesssim h_{K}^{-1/2}p_K^{}\|\bm v\|_{0,K}~\text{and}~
    |\bm v|_{1, K}\lesssim h_{K}^{-1}p_K^{2}\|\bm v\|_{0,K}.
\end{align}
\end{lem}
The following theorem shows the  coerciveness of $a_h(\cdot,\cdot)$ in $\b V_{h  }$, which guarantees the well-posedness of the discrete problem (\ref{2.10aa}).
\begin{theorem}\label{t3.2}
For sufficiently large $\eta_1$ and $\eta_2$, we have
\begin{align}\label{1}
    a_h(\bm v,\bm v)\gtrsim\|\bm v\|_{h}^2,~\forall \bm v\in \b V_{h  }.
\end{align}
\end{theorem}
\begin{proof}
Let us pick up any $\bm v\in \b V_{h  }$ we deduce
\begin{align}\label{1}
    a_h(\bm v,\bm v)&\ge\sum_{K\in\mathcal T_{h}} \left(\|\mathrm{curl}^2 \bm v\|_{0, K}^2+\|\bm v\|_{0, K}^2\right)+\sum_{\bm f\in \mathcal{E}_h}\left(\eta_1h_fp^2_f\|\llbracket \mathrm{curl}\bm v\rrbracket \|_{0,f}^2+
     \eta_2 h_f^{-3}p^6_f\|\llbracket \bm v\rrbracket \|_{0,f}^2\right)
    \nonumber\\
    &~~~-\rho_1^{-1}\sum_{\bm f\in \mathcal{E}_h}h_fp^{-2}_f\|\{\mathrm{curl}^2\bm v\}\|_{0,f}^2-\rho_1\sum_{\bm f\in \mathcal{E}_h}h_f^{-1}p^2_f\|\llbracket \mathrm{curl}\bm v\rrbracket \|_{0,f}^2\nonumber\\
    &~~~-\rho_2^{-1}\sum_{\bm f\in \mathcal{E}_h}h_f^3p^{-6}_f\|\{\mathrm{curl}^3\bm v\}\|_{0,f}^2-\rho_2\sum_{\bm f\in \mathcal{E}_h}h_f^{-3}p^6_f\|\llbracket \bm v\rrbracket \|_{0,f}^2
\end{align}
with $\rho_1$ and $\rho_2$ to be determined later.
Using  trace theorem and inverse estimate in Lemma 3.1, we have for $f= \overline{K_1}\cap \overline{K_2}$ 
\begin{align}\label{2.19}
h_fp_f^{-2}\|\{\mathrm{curl}^2\bm v\}\|_{0,f}^2&\le 0.5C_1\|\mathrm{curl}^2\bm v\|_{0, K_1\cup K_2}^2,\\
h_f^3p_f^{-6}\|\{\mathrm{curl}^3\bm v\}\|_{0,f}^2&\le C_2h_fp_f^{-2}\|\{\mathrm{curl}^2\bm v\}\|_{0,f}^2\le 0.5C_2C_1\|\mathrm{curl}^2\bm v\|_{0, K_1\cup K_2}^2.\label{2.20}
\end{align}
Then we deduce
\begin{align}\label{1}
    a_h(\bm v,\bm v)&\ge\sum_{K\in\mathcal T_{h}} \left(\|\mathrm{curl}^2 \bm v\|_{0, K}^2+\|\bm v\|_{0, K}^2\right)+\sum_{\bm f\in \mathcal{E}_h}\left(\eta_1h_f^{-1}p^2_f\|\llbracket \mathrm{curl}\bm v\rrbracket \|_{0,f}^2+
     \eta_2h_f^{-3}p^6_f \|\llbracket \bm v\rrbracket \|_{0,f}^2\right)
    \nonumber\\
    &~~~-2\rho_1^{-1}C_1\sum_{K\in\mathcal T_{h}}\|\mathrm{curl}^2\bm v\|_{0, K}^2-\rho_1\sum_{\bm f\in \mathcal{E}_h}h_f^{-1}p^2_f\|\llbracket \mathrm{curl}\bm v\rrbracket \|_{0,f}^2\nonumber\\
    &~~~-2\rho_2^{-1}C_1C_2\sum_{K\in\mathcal T_{h}}\|\mathrm{curl}^2\bm v\|_{0, K}^2-\rho_2\sum_{\bm f\in \mathcal{E}_h}h_f^{-3}p^6_f\|\llbracket \bm v\rrbracket \|_{0,f}^2.
\end{align}
Hence
\begin{align}\label{2.20}
    a_h(\bm v,\bm v)&\ge\sum_{K\in\mathcal T_{h}} \left((1-2\rho_1^{-1}C_1-2\rho_2^{-1}C_1C_2)\|\mathrm{curl}^2 \bm v\|_{0, K}^2+\|\bm v\|_{0, K}^2\right)
    \nonumber\\
    &~~~+\sum_{\bm f\in \mathcal{E}_h}\left(h_f^{-1}(\eta_1-\rho_1)p_f^2\|\llbracket \mathrm{curl}\bm v\rrbracket \|_{0,f}^2+
   (\eta_2-\rho_2)h_f^{-3}p_f^6\|\llbracket \bm v\rrbracket \|_{0,f}^2\right)
\end{align}
Choose the large $\rho_1$,   $\rho_2$,   $\eta_1$ and $\eta_2$ such that $2\rho_1^{-1}C_1+2\rho_2^{-1}C_1C_2<1$, $\eta_1>\rho_1$ and $\eta_2>\rho_2$.
Finally the inequality (\ref{2.20}) together with (\ref{2.19})  yields the conclusion.
\end{proof}

In the following theorem, we prove the error estimate of the discrete problem  (\ref{2.10aa}) by establishing the  interpolation error estimates in DG-norm.
\begin{theorem}\label{t3.3}
Assume that $\{\mathcal T_{h}\}$ is a family of  nonconforming  meshes. 
Let $\bm w\in\b H_0(\mathrm{curl}^2,\Omega)$ and $A_K(\bm w)<\infty$ for any $K\in \mathcal T_h$ then
  \begin{align}\label{2.22}
  \|\bm w- \bm w_{h  }\|_{h}&\lesssim \left(\sum_{K\in\mathcal T_{h}}(h_K^{\min(p_K+1,s_K)-2}p_K^{3.5-s_K}A_K(\bm w))^2\right)^{1/2},\quad p:={\min_{K\in\mathcal T_h}p_K}\ge2
  \end{align}
  where   $A_K(\bm w):=\|\bm w\|_{s_K, K}$ if $\|\bm w\|_{s_K, K}<\infty$ for  $s_K\ge 4$,
  otherwise $A_K(\bm w):=\|\bm w\|_{3, K}+\|\mathrm{curl}^3\bm w\|_{1/2+\delta, K}$  $(\delta>0)$ and $s_K=3$.
\end{theorem}
\begin{proof}{Let $\Pi_{p}$ be the global projection operator  which equals $\Pi_{p}^K$ on $K$.} First of all we shall prove
\begin{align}\label{2.22sa}
  \|\bm w-\Pi_{p}\bm w\|_{h}&\lesssim \left(\sum_{K\in\mathcal T_{h}}(h_K^{\min(p_K+1,s_K)-2}p_K^{3.5-s_K}A_K(\bm w)\right)^{1/2},\quad p_K,s_K\ge2.
\end{align}
Let us denote $E_h(\bm w)=\bm w-\Pi_{p}\bm w$ then
\begin{align*}
\|E_h(\bm w)\|_{h}^2&\lesssim\sum_{K\in\mathcal T_{h}} \left(\|\mathrm{curl}^2 E_h(\bm w)\|_{0, K}^2+\| E_h(\bm w)\|_{0, K}^2\right)
    +\sum_{\bm f\in \mathcal{E}_h} h_f^3p_f^{-6}\|\{\mathrm{curl}^3 E_h(\bm w)\}\|_{0,f}^2\\
    &~~~+\sum_{\bm f\in \mathcal{E}_h}\left( p_f^6h_f^{-3}\|\llbracket E_h(\bm w)\rrbracket \|_{0,f}^2+
    p_f^2h_f^{-1}\|\llbracket \mathrm{curl}E_h(\bm w)\rrbracket \|_{0,f}^2+p_f^{-2}h_f\|\{\mathrm{curl}^2E_h(\bm w)\}\|_{0,f}^2\right)\\
    &:=I_1+I_2+I_3.
\end{align*}
We first estimate $I_1$. Using the error estimate of $\Pi^K_{p}$ in Lemma 2.1, we have
\begin{align*}
\| E_h(\bm w)\|_{0, K}+\|\mathrm{curl}^2 E_h(\bm w)\|_{0, K}\lesssim h_K^{\min(p_K+1,s_K)-2}p_K^{2-s_K}\|\bm w\|_{s_K, K}.
\end{align*}
Let $f= \overline{K}_1\cap \overline{K}_2$ and {$s_f:=\min(s_{K_1},s_{K_2})$}. For $p_f\ge3$ we have from Lemma 2.1
\begin{align*}
|\mathrm{curl}^3 E_h(\bm w)|_{1, K_1\cup K_2}+h_f^{-1}p_f|\mathrm{curl}^3 E_h(\bm w)|_{0, K_1\cup K_2}
\lesssim h_f^{\min(p_f+1,s_f)-4}p_f^{4-s_f}\|\bm w\|_{s_f, K_1\cup K_2}.
\end{align*}
Then we  estimate $I_2$:
For $p_f\ge2$ we have for $i=1,2$
\begin{align*}
\|\mathrm{curl}^3 E_h(\bm w)|_{K_i}\|_{0,f}^2=
\|\mathrm{curl}^3 \bm w|_{K_i}\|_{0,f}^2\lesssim \|\mathrm{curl}^3 \bm w\|_{1/2+\delta, K_1\cup K_2}^2
\end{align*}
while for $p_f\ge3$ from Lemma A.3 in \cite{prudhomme} we have for $i=1,2$
\begin{align*}
\|\mathrm{curl}^3 E_h(\bm w)|_{K_i}\|_{0,f}&\lesssim h_f^{-1/2}\|\mathrm{curl}^3 E_h(\bm w)\|_{0, K_1\cup K_2}+(\|\mathrm{curl}^3 E_h(\bm w)\|_{0, K_1\cup K_2}|\mathrm{curl}^3 E_h(\bm w)|_{1, K_1\cup K_2})^{1/2}\\
&\lesssim (h_f^{\min(p_f+1,s_f)-3.5}p_f^{3-s_f}+h_f^{\min(p_f+1,s_f)-3.5}p_f^{3.5-s_f})\|\bm w\|_{s_f, K_1\cup K_2}\\
&\lesssim h_f^{\min(p_f+1,s_f)-3.5}p_f^{3.5-s_f}\|\bm w\|_{s_f, K_1\cup K_2}.
\end{align*}
We can estimate $I_3$ similarly:  
\begin{align}\label{eh}
&\|E_h(\bm w)|_{K_i}\|_{0,f}+h_fp_f^{-1}\|\mathrm{curl}(E_h(\bm w))|_{K_i}\|_{0,f}+h_f^2p_f^{-2}\|\mathrm{curl}^2(E_h(\bm w))|_{K_i}\|_{0,f}\nonumber\\
&~~~\lesssim  h_f^{\min(p+1,s_f)-0.5}p_f^{0.5-s_f}\|\bm w\|_{s_f, K_1\cup K_2}\text{for $i=1,2$}.
\end{align}
Hence we have proved \eqref{2.22sa}.
Using (\ref{2.11a}) and Theorem 3.1  we have
\begin{eqnarray*}
 &&\|\bm w_{h  }-\bm v_{h  }\|_{h}^2\lesssim a_h(\bm w_{h  }-\bm v_{h  },\bm w-\bm v_{h  })\lesssim \|\bm w_{h  }-\bm v_{h  }\|_{h}\|\bm w-\bm v_{h  }\|_{h},\quad \forall\bm v_{h  }\in  {\b V}_h.
\end{eqnarray*}
Then
\begin{eqnarray}\label{2.22a}
 &&\|\bm w-\bm w_{h  }\|_{h}\lesssim \inf_{\bm v_{h  }\in  {\b V}_h}\|\bm w-\bm v_{h  }\|_h.
\end{eqnarray}
This together with \eqref{2.22sa}  yields the conclusion.
\end{proof}

\begin{remark}
The error estimate \eqref{2.22sa}   on the general polygonal $(d = 2)$ or polyhedral $(d = 3)$ meshes  can be proved by using a triangle or a tetrahedron to cover the polygonal or polyhedral element, like the way in Lemma 23 in \cite{cangiani} or Lemma 4.12 in \cite{dong}.
\end{remark}

\begin{remark}
The    convergence rate is optimal with respect to $h$ but the  convergence rate in polynomial degree $p$ is not optimal under nonconforming mesh. The convergence rate can be improved under conforming meshes. 
The utilization of the $H^1$-conforming  element interpolation in \cite{babuska1} can yield the following theorem.
\end{remark}

\begin{theorem}\label{t3}Assume that $\{\mathcal T_{h}\}$ is a family of conforming meshes. Let $\bm w\in\b H_0(\mathrm{curl}^2,\Omega)$ and $A_\Omega(\bm w):=(\sum_{K\in\mathcal T_h}A_K(\bm w)^2)^{1/2}<\infty$ with $\min_{K\in\mathcal T_h}s_K:=s$ then
\begin{align}\label{2.22s}
  ||\bm w- \bm w_{h  }||_{h}&\lesssim  h^{\min(p+1,s)-2}p^{3-s}A_{\Omega}(\bm w),\quad p\ge2.
  \end{align}
\end{theorem}
\begin{proof}
Let $Q_{h  }$ be the $C^0$-conforming finite element space defined in \cite{babuska1} {with the polynomial degree $p$ and the mesh size $h$}. There is  a projection $P_{h  }:\b H^1(\Omega)\to Q_{h  }\times Q_{h  }$ (see Theorem 4.8 in \cite{babuska1}) such that
 \begin{align}\label{h1}
\|P_{h  }\bm v-\bm v\|_{1,\Omega}\lesssim  h^{\min(p+1,l)-1}p^{1-l}\|\bm v\|_{l, \Omega},\quad p,l\ge1.
\end{align}
Since
\begin{align*}
|\mathrm{curl} (P_{h  }\bm w-\bm w)|_{1,K}&\le |\mathrm{curl}P_{h  }\bm w-P_{h  }\mathrm{curl}\bm w|_{1,K}+|P_{h  }\mathrm{curl}\bm w-\mathrm{curl}\bm w|_{1,K}\nonumber\\
&\lesssim p_K^2h_K^{-1}\|\mathrm{curl}P_{h  }\bm w-\mathrm{curl}\bm w+\mathrm{curl}\bm w-P_{h  }\mathrm{curl}\bm w\|_{0,K}+\| P_{h  }\mathrm{curl}\bm w-\mathrm{curl}\bm w\|_{1,K},
\end{align*}
we have 
\begin{align}\label{c-2}
\left(\sum_{K\in\mathcal T_h}|\mathrm{curl} (P_{h  }\bm w-\bm w)|_{1,K}^2\right)^{1/2}
&\lesssim  h^{\min(p+1,s)-2}p^{3-s}\|\bm w\|_{s, \Omega}.
\end{align}
{Similarly as above, we have
\begin{align}\label{c-1}
\left(\sum_{K\in\mathcal T_h}\|\mathrm{curl}^2 (P_{h  }\bm w-\bm w)\|_{1,K}^2+h_K^2p_K^{-4}\|\mathrm{curl}^3 (P_{h  }\bm w-\bm w)\|_{1,K}^2\right)^{1/2}\lesssim  h^{\min(p+1,s)-3}p^{5-s}A_{\Omega}(\bm w).
\end{align}}
Denoted by $\widetilde E_h(\bm w):=P_{h  }\bm w-\bm w$,
since for $i=1,2$ 
\begin{align*}
\|\mathrm{curl}^2 \widetilde{E}_h(\bm w)|_{K_i}\|_{0,f}&\lesssim h_f^{-1/2}\|\mathrm{curl}^2 \widetilde{E}_h(\bm w)\|_{0, K_1\cup K_2}+(\|\mathrm{curl}^2 \widetilde{E}_h(\bm w)\|_{0, K_1\cup K_2}|\mathrm{curl}^2 \widetilde{E}_h(\bm w)|_{1, K_1\cup K_2})^{1/2}
\end{align*}
we have
\begin{align}\label{c1}
\left(\sum_{\bm f\in \mathcal{E}_h}\|\{\mathrm{curl}^2 \widetilde{E}_h(\bm w)\}\|_{0,f}^2\right)^{1/2}&\lesssim 
 (h^{\min(p+1,s)-2.5}p^{3-s}+h^{\min(p+1,s)-2.5}p^{4-s})\|\bm w\|_{s, \Omega}\nonumber\\
&\lesssim h^{\min(p+1,s)-2.5}p^{4-s}\|\bm w\|_{s, \Omega}
\end{align}
and similarly for  $i=1,2$
\begin{align}\label{c2}
\left(\sum_{\bm f\in \mathcal{E}_h}\|\{\mathrm{curl}^3 \widetilde{E}_h(\bm w)\}\|_{0,f}^2\right)^{1/2}\lesssim h^{\min(p+1,s)-3.5}p^{6-s}A_{\Omega}(\bm w),\\
\left(\sum_{\bm f\in \mathcal{E}_h}\|\llbracket\mathrm{curl} \widetilde{E}_h(\bm w)\rrbracket\|_{0,f}^2\right)^{1/2}\lesssim h^{\min(p+1,s)-1.5}p^{2-s}\|\bm w\|_{s, \Omega}.\label{c3}
\end{align}
 The  estimates \eqref{c-2}-\eqref{c3} give
\begin{align}\label{i}
\|\widetilde{E}_h(\bm w)\|_{h}\lesssim h^{\min(p+1,s)-2}p^{3-s}A_{\Omega}(\bm w).
\end{align}
The assertion can be done by the  inequality \eqref{2.22a}. 
\end{proof}

\section{IPDG method for the quad-curl eigenvalue problem}
In this section we restrict our analysis on the case $P(K)=(P_{p}(K))^d$ on a simplex $K$. 
The quad-curl eigenvalue problem reads: Find $\bm u\in\b H_0(\mathrm{curl}^2,\Omega)$ and $\widetilde{p}\in H^1_0(\Omega)$ such that
\begin{align}\label{2.10a}
\begin{aligned}
(\mathrm{curl}^2\bm u,\mathrm{curl}^2\bm{v})+(\nabla \widetilde{p},\bm{v})&=\lambda( \bm u,\bm v),~~\forall \bm {v}\in \b H_0(\mathrm{curl}^2,\Omega),\\
(\bm u,\nabla q)&=0,~~\forall q\in H_0^1(\Omega).
\end{aligned}
\end{align}
One readily verifies that
\begin{align}\label{2.10}
\begin{aligned}
\wt a_h(\bm u,\mathbf{v})+(\nabla \widetilde{p},\mathbf{v})&=\lambda( \bm u,\bm v),~~\forall \bm {v}\in \b V_{h  },
\end{aligned}
\end{align}
where
$$\wt a_h(\bm u_{ },\bm{v}_{ })= a_h(\bm u_{ },\bm{v}_{ })-(\bm u_{ },\bm{v}_{ }).$$
 {We introduce the following function spaces:}
\begin{align}
\b X &= \{\bm v \in  \b H_0(\mathrm{curl}, \Omega) |(\bm v,\nabla q)=0,  \forall q\in H^1_0(\Omega)\},\\
U_{h  }&=\{q_h\in  H^1_0(\Omega)|q_h|_ K\in P_ {p+1}(K),~\forall  K\in\mathcal T_{h}\},\label{4.4}\\
\b X_{h  } &= \{\bm v_{h  } \in {\b V}_h  |(\bm v_{h  },\nabla q)=0,  \forall q\in U_{h  }\}.
\end{align}
The discrete form of the   eigenvalue problem (\ref{2.10a}) is given by:
Find $(\lambda_h,\b u_h,\widetilde p_h)\in \mathbb{R}\times \b V_{h  }\times  U_{h  }$
 with $\b u_h\neq0$ such that
\begin{align}\label{2.9}
\begin{aligned}
\wt a_h(\bm u_{h  },\bm{v}_{h  })+(\nabla \widetilde{p}_h,\bm{v}_{h  })&=\lambda_h( \bm u_{h  },\bm v_{h  }),~~\forall \bm{v}_{h  }\in  \b V_{h  },\\
(\bm u_{h  },\nabla q)&=0,~~\forall q\in U_{h  }.
\end{aligned}
\end{align}
To analyze the convergence of the discretization \eqref{2.9},
we consider the   source problem with div-free constraint: Find $T\bm f\in\b H_0(\mathrm{curl}^2,\Omega)$ and $S\bm f\in H^1_0(\Omega)$ such that
\begin{align}\label{p2}
\begin{aligned}
(\mathrm{curl}^2T\bm f,\mathrm{curl}^2\bm{v})+( \nabla S\bm f, \bm v)&=( \bm f,\bm v),~~\forall \bm{v}\in \b H_0(\mathrm{curl}^2,\Omega),\\
( \nabla q, T\bm f)&=0,~~\forall q\in H^1_0(\Omega).
\end{aligned}
\end{align}
Its DG discretization is to seek $T_{h  }\bm f\in \b V_{h  }$ and $S_h\bm f\in U_{h  }$:
\begin{align}\label{p3}
\begin{aligned}
\wt a_h(T_{h  }\bm f,\bm{v}_{h  })+( \nabla S_h\bm f, \bm v_{h  })&=( \bm f,\bm v_{h  }),~~\forall \bm{v}_{h  }\in \b V_{h  },\\
(  \nabla q, \bm T_{h  }\bm f)&=0,~~\forall q\in U_{h  }.
\end{aligned}
\end{align}
{It is easy to verify that $S\bm f=S_h \bm f=0$ for $\bm f \in \b X$.}

\begin{proposition}[Propositions 4.5  in \cite{houston2005}]\label{p1}
Let $\bm v_{h  }\in  \b V_{h  }$. Then there is  $\bm v_{h  }^c\in \b H_0(\mathrm{curl},\Omega)\cap\b V_{h  }$ such that
\begin{align}
h_f^{-2}\|\bm v_{h  }-\bm v_{h  }^c\|^2+\|\mathrm{curl}_h(\bm v_{h  }-\bm v_{h  }^c)\|^2\lesssim \sum_{f\in \mathcal{E}_h}
    h_f^{-1}\|\llbracket \bm v_{h  }\rrbracket \|^2_{0,f}.\label{2.32}
\end{align}
where  $\|\mathrm{curl}_h(\cdot)\|:=\|(\sum_{K\in \mathcal T_h}\|\mathrm{curl}(\cdot)\|_{0,K}^2)^{1/2}$ and
 the hidden constant is independent of the mesh size $h$ and the polynomial degree $p$.
    \end{proposition}
\begin{remark}The above conclusion is valid on nonconforming meshes since they can be conformed by adding some edges ($d=2$) or faces($d=3$). The fact that the hidden constant in \eqref{2.32} is independent of the mesh size $h$ is verified in \cite{houston2005}. However, its independence on the polynomial degree $p$ is not verified yet. Here we give an argument for the case of $\b H(\mathrm{curl})$-conforming rectangular   element method   as follows.

Let $\{\varphi_{K,e}^i\}_{i=1}^{4p}$ and $\{\varphi_{K,b}^i\}_{i=1}^{2p^2-2p}$ be the edge-based basis functions and
the cell-based basis functions   on $Q^{p-1,p}(K)\times Q^{p,p-1}(K)$(see Appendix A), respectively. Any $\bm v\in Q^{p-1,p}(K)\times Q^{p,p-1}(K)$  can be written as
\begin{align}
\bm v=\sum_{e\in\mathcal{E}(K)}\sum_{i=1}^{4p}\bm v_{K,e}^i\varphi_{K,e}^i
+\sum_{i=1}^{2p^2-2p}\bm v_{K,b}^i\varphi_{K,b}^i.\label{a1}
\end{align}
Let $ {\bm v}^c\in  \b H_0(\mathrm{curl},\Omega)\cap\b V_{h  }$ be the function whose edge moments are
\begin{align*}
\overline{\bm v}_{K,e}^i=\left\{\begin{tabular}{cc}
                                  $\frac{1}{2}\sum_{e\subset K'} {\bm v}_{K',e}^i$&\text{if} $e\not\subset \partial\Omega$\\
                                  $0$&\text{if} $e\subset \partial\Omega$
                                \end{tabular}
\right.
\end{align*}
for $i=1,\cdots,4p$ and whose  cell moments are
\begin{align*}
\overline{\bm v}_{K,b}^i={\bm v}_{K,b}^i
\end{align*}
for $i=1,\cdots,2p^2-2p$.
According to the transformation $\bm v\cdot F_K=B_K^{-T}\widehat{\bm v}$,
\begin{align}
\|\bm v\|_{0,K}^2\lesssim \|\widehat{\bm v}\|_{0,\widehat K}^2,\quad \|\mathrm{curl}\bm v\|_{0,K}^2\lesssim h_K^{-2}\|\mathrm{curl}\widehat{\bm v}\|_{0,\widehat K}^2.\label{a2}
\end{align}
Note that $\widehat{\bm v}-\widehat{\bm v}^c\in Q^{1,p}([-1,1]^2)+Q^{p,1}([-1,1]^2)$. By the orthogonality of basis functions we have
\begin{align}
\|\mathrm{curl}(\widehat{\bm v}-\widehat{\bm v}^c)\|^2_{0,\widehat{K}} &\lesssim  \sum_{e\subset\partial K}\sum_{i=1}^{N_e}(\bm v_{K,e}^i-\overline{\bm v}_{K,e}^i)^2\lesssim \sum_{e\subset\partial K}\sum_{K'\supset e}\sum_{i=1}^{N_e}(\bm v_{K,e}^i-{\bm v}_{K',e}^i)^2\nonumber\\\label{a4}
&\lesssim \sum_{e\subset\partial K}h_e^{}\int_{e} |{(\bm v|_K-\bm v|_{K'})}\cdot \tau|^2ds
\end{align}
and  similarly by  \eqref{5.2} and \eqref{5.3} in Appendix A
\begin{align}
\|\widehat{\bm v}-\widehat{\bm v}^c-\Pi(\widehat{\bm v}-\widehat{\bm v}^c)\|^2_{0,\widehat{K}}+\|\Pi(\widehat{\bm v}-\widehat{\bm v}^c)\|^2_{0,\widehat{K}}\lesssim \sum_{e\subset\partial K}\sum_{i=1}^{N_e}(\bm v_{K,e}^i-\overline{\bm v}_{K,e}^i)^2
\lesssim \sum_{e\subset\partial K}h_e^{}\int_{e} |{(\bm v|_K-\bm v|_{K'})}\cdot \tau|^2ds.\label{a4.1}
\end{align}
Finally  we have from the above two estimates and \eqref{a2}
\begin{align}
 h_K^{-2}\|\bm v- {\bm v}^c\|^2_{0,K}+\|\mathrm{curl}(\bm v-{\bm v}^c)\|^2_{0,K} \lesssim \sum_{e\subset\partial K}h_K^{-1}\|\llbracket \bm v\rrbracket \|^2_{0,e}\label{a4}.
\end{align}
Then \eqref{2.32} follows.
\end{remark}

The Hodge operator is a  useful tool in our error analysis.
It is defined as $ H \bm g\in \c0$ and $\rho\in H^1_0(\Omega)$ for $\bm g\in {\b X}_h$ such that
\begin{align}\label{2.5s}
\begin{aligned}
    &(\mathrm{curl} H \bm g,\mathrm{curl} \bm v)+  ( \nabla \rho, \bm v)=(\mathrm{curl}\bm g,\mathrm{curl}\bm v),~~\forall \bm{v}\in \mathbf{H}_0(\mathrm{curl},\Omega),\\
     &( \nabla q, H \bm g)=0,~~\forall q\in H^1_0(\Omega).
\end{aligned}
\end{align}

 {We introduce the following   curl-conforming finite element spaces \cite{nedelec}}
\begin{align}
\widetilde{\b V}_{h  } &= \{\bm v_{h  } \in \b H_0(\mathrm{curl}, \Omega) |\bm v_{h  }|_ K \in (P_{p}(K))^d,  K\in\mathcal T_{h}\},\\
\widetilde{\b X}_{h  } &= \{\bm v_{h  } \in  \widetilde{\b V}_{h  }\text{~and~}(\bm v_{h  },\nabla q)=0,  \forall q\in U_{h  }\},
\end{align}
and give the corresponding  interpolation error estimates in virtue of   Lemma 3.1  in \cite{boffi3}.
\begin{lem} {Let $r_{h  }$ be the  edge element interpolation associated with $\widetilde{\b V}_{h  }$. If $\bm v\in \b X\subset \b H^{r_0}(\Omega)$ for some $r_0\in (1/2,1]$ and $\mathrm{curl}\bm v\in \mathrm{curl}\widetilde{\b V}_{h  }$ then}
\begin{align}\label{b1}
\|r_{h  } \bm v- \bm v\|_{}\lesssim
 h^{r_0}p^{-1/2}(\|\bm v\|_{{r_0}}+ \|\mathrm{curl}\bm v\|_{}).
\end{align}
    \end{lem}
\begin{proof}Denote $\widehat{\bm v}=(\nabla F)^{-T}\bm v\circ F$ where $F$ is the affine mapping between $K$ and its reference element $\widehat K$.
Lemma 3.1 (Inequality (3.22)) in \cite{boffi3} together with Theorem 5.3 in \cite{demkowicz} and  Theorem 4.1 in \cite{bespalov} gives
\begin{align}\label{b1}
\|\widehat r_{\widehat K  } \widehat{\bm v}-  \widehat{\bm v}\|_{0,\widehat K }\lesssim
 p^{-1/2}(\| \widehat{\bm v}\|_{{r_0},\widehat K }+ \|\mathrm{curl} \widehat{\bm v}\|_{0,\widehat K })
\end{align}
where $\widehat r_{\widehat K  }\widehat{\bm v}=(\nabla F)^{-T} r_{h  }\bm v\circ F$.
The conclusion can be deduced via the scaling argument.
\end{proof}

 \begin{lem}[Lemma 7.6 in \cite{monk}] Let $\bm v_{h  }\in \widetilde{\b X}_{h  }$ then $H\bm v_{h  }\in \b X$ such that $\mathrm{curl}\bm v_{h  }=\mathrm{curl}H\bm v_{h  }$ and
\begin{align}
\|H\bm v_{h  } \|_{r_0}&\lesssim\|\mathrm{curl}\bm v_{h  } \|_{},\label{4.15}\\
 \|H\bm v_{h  }-\bm v_{h  }\|_{}&\le  \|r_{h  } H\bm v_{h  }- H\bm v_{h  }\|_{}.
 \label{4.16}
\end{align}
    \end{lem}
\begin{lem}
For any $\bm v_{h  }\in\b V_{h  }$ it is valid that
\begin{align}\label{2.1aa}
  \sum_{K\in\mathcal T_{h}}\|\mathrm{curl}\bm v_{h  }\|_{0,K}^2\lesssim\sum_{K\in\mathcal T_{h}}\|\mathrm{curl}^2\bm v_{h  }\|_{0,K}^2+\sum_{f\in \mathcal{E}_h}\left(h_f^{-1}\|\llbracket \mathrm{curl}\bm v_{h  }\rrbracket \|_{0,f}^2+h_f^{-2}p_f^4\|\llbracket \bm v_{h  }\rrbracket \|_{0,f}^2\right).
\end{align}
\end{lem}
\begin{proof}This proof follows the argument of Lemma 4.3 in \cite{chen}. We introduce the auxiliary problem: Find {$\sigma_h\in U_h$} such that
\begin{align*}
  (\nabla\sigma_h,\nabla s)=((\mathrm{curl}_h \bm v_h)^c,\nabla s),\quad\forall s\in U_h.
\end{align*}
Then using Proposition 4.1 and  Lemma 3.1 we have
\begin{align*}
  (\nabla\sigma_h,\nabla\sigma_h)&=((\mathrm{curl}_h \bm v_h)^c-\mathrm{curl}_h\bm v_h+\mathrm{curl}_h\bm v_h,\nabla \sigma_h)\\
  &\lesssim  |\sigma_h|_1\left(\sum_{f\in \mathcal{E}_h}
    h_f^{-1}\|\llbracket \mathrm{curl}_h\bm v_h\rrbracket \|^2_{0,f}\right)^{1/2}+|\sum_{f\in \mathcal{E}_h}\int_f\mathrm{div}(\llbracket \bm v_h\rrbracket )\sigma_hds|\\
&\lesssim  |\sigma_h|_1\left(\sum_{f\in \mathcal{E}_h}
    h_f^{-1}\|\llbracket \mathrm{curl}_h\bm v_h\rrbracket \|^2_{0,f}+\sum_{f\in \mathcal{E}_h}h_f^{-2}p_f^4\|\llbracket \bm v_h\rrbracket \|_{0,f}^2\right)^{1/2}.
\end{align*}
Note that  $\nabla\sigma_h-(\mathrm{curl}_h \bm v_h)^c\in\widetilde{\b X}_h$. We have from Proposition 4.1
\begin{align*}
  \|(\mathrm{curl}_h \bm v_h)^c-\nabla \sigma_h\|&\lesssim \|\mathrm{curl}((\mathrm{curl}_h \bm v_h)^c-\nabla \sigma_h)\|\\
  &\lesssim \|\mathrm{curl}_h((\mathrm{curl}_h \bm v_h)^c-\mathrm{curl}_h \bm v_h+\mathrm{curl}_h \bm v_h)\|\\
    &\lesssim \left(\sum_{f\in \mathcal{E}_h}
    h_f^{-1}\|\llbracket \mathrm{curl}_h\bm v_h\rrbracket \|^2_{0,f}\right)^{1/2}+\|\mathrm{curl}^2_h  \bm v_h\|.
\end{align*}
The combination of the above two inequalities and Proposition 4.1 yields the conclusion.
\end{proof}
The following result guarantees  the well-posedness of the   problem (\ref{p3}).
\begin{lem}
There holds the following discrete poinc\'{a}re inequality
\begin{eqnarray}\label{2.22}
 &&\|\bm v_{h  }\|_{h}\lesssim |\bm v_{h  }|_{h},\quad \forall\bm v_{h  }\in\b X_{h  }.
\end{eqnarray}
\end{lem}
\begin{proof} Let $\bm v_{h  }^c$ be defined as in Proposition 4.1 for any $\bm v_{h  }\in\b X_{h  }$. We have the Helmholtz decomposition
$\bm v_{h  }^c=\bm v_{h  }^0\bigoplus\nabla e_0$ with $\bm v_{h  }^0\in \widetilde{\b X}_{h  }$ and $ e_0\in U_{h  }$. 
 Due to the fact $(H\bm v_{h  }^0- {\bm v}_{h  },\nabla e_0)=0$ we have
\begin{align*}
 (H\bm v_{h  }^0- {\bm v}_{h  },H\bm v_{h  }^0- {\bm v}_{h  })
 = (H\bm v_{h  }^0- {\bm v}_{h  },H\bm v_{h  }^0-{\bm v}_{h  }^0+{\bm v}_{h  }^c- {\bm v}_{h  })
\end{align*}
which together with   (\ref{2.32}), Lemma 4.2 and (\ref{b1}) leads to
\begin{align}\label{3.16aa}
 \|H\bm v_{h  }^0- {\bm v}_{h  }\|
 &\lesssim \|H\bm v_{h  }^0-{\bm v}_{h  }^0\|+\|{\bm v}_{h  }^c- {\bm v}_{h  }\| \nonumber\\
 &\lesssim h^{r_0}p^{-1/2}\|\mathrm{curl}{\bm v}^c_{h  }\|_{}+ \left(\sum_{f\in \mathcal{E}_h}
    h_f^{-1}\|\llbracket \bm v_{h  }\rrbracket \|^2_{0,f}\right)^{1/2}\nonumber\\
 &\lesssim h^{r_0}p^{-1/2}(\sum_{K\in \mathcal T_{h}}\|\mathrm{curl}\bm v_{h  }\|_{0,K}^2)^{1/2}+hp^{-3}|\bm v_{h  }|_{h}.
\end{align}
This combines with  (\ref{2.32}) and (\ref{2.1aa}) to get
\begin{align}\label{2.22}
 \|\bm v_{h  }\|&\lesssim\| H\bm v^0_{h  }\|_{}+h^{r_0}p^{-1/2} (\sum_{K\in\mathcal T_{h}}\|\mathrm{curl}\bm v_{h  }\|_{0, K}^2)^{1/2}+hp^{-3}|\bm v_{h  }|_{h}\nonumber\\
 &\lesssim\|\mathrm{curl}{\bm v}^c_{h  }\|_{}+h^{r_0}p^{-1/2}|\bm v_{h  }|_{h}\nonumber\\
  &\lesssim (\sum_{K\in\mathcal T_{h}}\|\mathrm{curl}\bm v_{h  }\|_{0, K}^2)^{1/2}+hp^{-3}|\bm v_{h  }|_{h}+h^{r_0}p^{-1/2}|\bm v_{h  }|_{h}\nonumber\\
  &\lesssim |\bm v_{h  }|_{h}+h^{r_0}p^{-1/2}|\bm v_{h  }|_{h}.
\end{align}
This leads to the conclusion directly.
\end{proof}
The estimate \eqref{2.20} shows that $\widetilde a_h(\bm v,\bm v)\gtrsim |\bm v|_h^2$ for any $\bm v\in \b V_h$.
According to Lemma 4.4 we have the following.
\begin{corollary} It holds the a-priori estimate
\begin{eqnarray*}
\|(  T  -  T_h) \bm f\|_{h}\lesssim\inf\limits_{\bm{v}_h\in\mathbf V_h}\| {T} \bm f-\bm {v}_h\|_{h}+\inf\limits_{v_h\in  U_h}|S \bm f- v_h|_{1,\Omega}\rt 0.
 \end{eqnarray*}
\end{corollary}
\begin{proof}
The combination of \eqref{p2} and \eqref{p3} leads to
\begin{align*}
\wt a_h((T_{ }-T_{h  })\bm f,\bm{v}_{h  })+( \nabla (S-S_h)\bm f, \bm v_{h  })&=0,\quad\forall \bm{v}_{h  }\in \b V_{h  }.
\end{align*}
It follows that
\begin{align*}
\wt a_h(\bm v_h-T_{h  }\bm f,\bm v_h-T_{h  }\bm f)=\wt a_h(\bm v_h-T\bm f,\bm v_h-T_{h  }\bm f)-( \nabla (S-S_h)\bm f, \bm v_h-T_{h  }\bm f),\quad\forall \bm{v}_{h  }\in \b X_{h  }.
\end{align*}
Hence by  the   discrete poinc\'{a}re inequality \eqref{2.22} we have
\begin{align*}
\|\bm v_h-T_{h  }\bm f\|_h\lesssim |T\bm f-\bm v_h|_h+|(S-S_h)\bm f|_{1,\Omega},\quad\forall \bm{v}_{h  }\in \b X_{h  }.
\end{align*}
 {This together with the triangular inequality infers that
\begin{align*}
\|T_{ }\bm f-T_{h  }\bm f\|_h\lesssim \|T\bm f-\bm w_h\|_h+|(S-S_h)\bm f|_{1,\Omega}
\end{align*}
where $\bm w_h$ is as in \eqref{2.22a} with $\bm w:=T\bm f$.  Then the proof is finished.}
\end{proof}
 We can prove the following $hp$-version error estimates for of  the IPDG solution.
\begin{theorem}
{Assume one of the following regularities for the equations \eqref{p2}   is valid:}
\begin{align}\label{p4}
&\| T\bm f\|_{4,\Omega} \lesssim \|\bm f\|,\\
&\| T\bm f\|_{r_0,\Omega}+\|\mathrm{curl} T\bm f\|_{1+r_1,\Omega}+\|\mathrm{curl}^3 T\bm f\|_{1/2+\delta,\Omega}\lesssim \|\bm f\|, \text{for } r_0,r_1\in(1/2,1]\text{and }\delta>0.\label{p5}
\end{align}
Let $\{\mathcal T_{h}\}$ be a family of  conforming  meshes.
Let $\bm g\in \b L^2(\Omega)$ with $\mathrm{div} \bm g=0$, $T\bm g\in\b H_0(\mathrm{curl}^2,\Omega)$ and $A_K(T\bm g)<\infty$ for any $K\in \mathcal T_h$ then
{\begin{eqnarray}\label{d1}
\|T\bm g -T_{h  }\bm g\|\lesssim \varepsilon(h,p) h^{\min(p+1,s)-2}p^{3-s}A_\Omega(T\bm g)
\end{eqnarray}}
where $\varepsilon(h,p)=h^{r_0}p^{-1/2}$ under Assumption \eqref{p4} and $\varepsilon(h,p)=h^{\min(r_0,r_1)}$ under Assumption \eqref{p5}.
\end{theorem}
\begin{proof}
For $\bm g\in L^2(\Omega)$ with $\mathrm{div}\bm g=0$ we have $S_{h  }\bm g=S\bm g=0$. We derive 
\begin{align*}
\|T\bm g-T_{h  }\bm g\|^2&=\left(T\bm g-T_{h  }\bm g,  T\bm g-P_{h  }T\bm g+P_{h  }T\bm g-\bm (T_{h  }\bm g)^c+(T_{h  }\bm g)^c-T_{h  }\bm g\right)
\end{align*}
Using the Helmholtz decomposition
$P_{h  }T\bm g-\bm (T_{h  }\bm g)^c=\bm w_0^c\bigoplus\nabla e$ with $\bm w_0^c\in \b X_{h  }$ and $ e\in U_{h  }$, we have from \eqref{2.32} in Proposition 4.1
\begin{align}\label{2.41ss}
\|T\bm g-T_{h  }\bm g\|^2\lesssim \left(\|T\bm g-P_{h  }T\bm g\|+\|(T_{h  }\bm g)^c-T_{h  }\bm g\|\right)\|T\bm g-T_{h  }\bm g\|+|(T\bm g-T_{h  }\bm g,\bm w^c_0)|.
\end{align}
Note that $S_{h  }\bm g=S\bm g=S_{h  }H\bm w^c_0=SH\bm w^c_0=0$. 
Let $\Pi_{h  }^c$ be the orthogonal projection onto {${\b V}_{h  }\cap \b H(\mathrm{curl},\Omega)$} such that for any $\bm v\in\b H_0(\mathrm{curl}^2,\Omega)$
\begin{align}
a_h(\bm v-\Pi_{h  }^c\bm v,\bm q)=0,~\forall\bm q\in {\b V}_{h  }\cap \b H(\mathrm{curl},\Omega).
\end{align}
If Assumption \eqref{p4} holds then by \eqref{i}  
\begin{align*}
|T H\bm w^c_0-\Pi_h^cT H\bm w^c_0|_h\lesssim  \|T H\bm w^c_0-P_hT H\bm w^c_0\|_h\lesssim h^{}p^{-1}\| T\bm H\bm w^c_0\|_{4,\Omega}\lesssim \varepsilon(h,p)\| \bm H\bm w^c_0\|_{}.
\end{align*}
If Assumption \eqref{p5} holds then the similar argument as those in Theorem 5.6 of \cite{chen} shows
\begin{align*}
|T H\bm w^c_0-\Pi_h^cT H\bm w^c_0|_h&\lesssim h^{\min(r_0,r_1)}(\| T\bm H\bm w^c_0\|_{r_0,\Omega}+\|\mathrm{curl} T\bm H\bm w^c_0\|_{1+r_1,\Omega}+\|\mathrm{curl}^3 T\bm H\bm w^c_0\|_{1/2+\delta,\Omega})\\
&\lesssim \varepsilon(h,p)\| \bm H\bm w^c_0\|_{}.
\end{align*}
Using Lemmas 4.2 and 4.1, the third term at the right-side hand of \eqref{2.41ss} is estimated as follows:
\begin{align}\label{2.45}
|(T\bm g-T_{h  }\bm g,\bm w^c_0)|&\le |(T\bm g-T_{h  }\bm g,\bm w^c_0- H\bm w^c_0)|+|(T\bm g-T_{h  }\bm g, H\bm w^c_0)|\nonumber\\
&\lesssim h^{r_0}p^{-1/2}\|\mathrm{curl}\bm w_0^c\|\|T\bm g-T_{h  }\bm g\|+|\widetilde a_h(T\bm g-T_{h  }\bm g, T H\bm w^c_0-\Pi_h^cT H\bm w^c_0)|\nonumber\\
&\lesssim h^{r_0}p^{-1/2}\|P_{h  }T\bm g-\bm (T_{h  }\bm g)^c\|_{\mathrm{curl}}\|T\bm g-T_{h  }\bm g\|+|T\bm g-T_{h  }\bm g|_h\varepsilon(h,p)\|H\bm w^c_0\|
\end{align}
where $\|H\bm w^c_0\|$ can be estimated by Lemmas 4.2 and 4.1:
\begin{align*}
\|H\bm w^c_0\|&\lesssim \|\bm w^c_0\|+h^{r_0}p^{-1/2}\|\mathrm{curl}(P_{h  }T\bm g-\bm (T_{h  }\bm g)^c)\|\\
&\lesssim \|(P_{h  }-I)T\bm g+T\bm g-T_h\bm g+T_h\bm g-\bm (T_{h  }\bm g)^c\|+h^{r_0}p^{-1/2}\|P_{h  }T\bm g-\bm (T_{h  }\bm g)^c\|_{\mathrm{curl}}.
\end{align*}
The substitution  of (\ref{2.45}) into the estimate (\ref{2.41ss}) gives
\begin{align}
\|T\bm g-T_{h  }\bm g\|&\lesssim h^{r_0}p^{-1/2}\|P_{h  }T\bm g-\bm (T_{h  }\bm g)^c\|_{\mathrm{curl}}+|T\bm g-T_{h  }\bm g|_h\varepsilon(h,p)+\|P_hT\bm g-T\bm g\|+\|T_h\bm g-(T_h\bm g)^c\|\nonumber\\
&\lesssim  \|T\bm g-T_{h  }\bm g\|_h\varepsilon(h,p)+\|P_hT\bm g-T\bm g\|+\|T_h\bm g-(T_h\bm g)^c\|\nonumber\\
&\quad +h^{r_0}p^{-1/2}(\|P_{h  }T\bm g-T_{ }\bm g\|_{h}+\|\mathrm{curl}_h(T_h\bm g-(T_h\bm g)^c)\|)\text{by {\eqref{2.1aa}}}\nonumber\\
&\lesssim\varepsilon(h,p)h^{\min(p+1,s)-2}p^{3-s}A_{\Omega}(T\bm g)+h^{r_0}\left(h_f^{-1}\|\llbracket  T_{h  } \bm g\rrbracket \|^2_{0,f}\right)^{1/2}\text{by \eqref{2.32}, Cor. 4.1 $\&$ Thm. 3.3}\nonumber\\
&\lesssim\varepsilon(h,p)h^{\min(p+1,s)-2}p^{3-s}A_{\Omega}(T\bm g)+h^{r_0+1}p^{-3}\|T_{h  } \bm g-T\bm g \|_{h}.
 \end{align}   
Then the  estimate   \eqref{d1} is obtained by using  Colloary 4.1 and Theorem 3.3.
\end{proof}


Let ${\mathcal{H}}$ be a sequence of  $(h,p)$ with $h/p$ converging to 0.
\begin{lem}(Discrete compactness property) Any sequence $\{\bm v_{h}\}_{(h,p)\in\mathcal{H}}$ with $\bm v_{h}\in \x_{h  }$ that is uniformly bounded w.r.t $\|\cdot\|_{h}$ contains a subsequence that converges strongly in $\b L^2(\Omega)$.
\end{lem}
\begin{proof}

  Let $\{\bm v_{h  }\}_{(h,p)\in\mathcal{H}}$  with $\|\bm v_{h  }\|_{h}<M$ for a positive constant $M$. It is trivial to  assume that the sequece $(h_i,p_i)\in\mathcal{H}$ converges to zero as $i\to\infty$. According to (\ref{3.16aa}), $\|H{\bm v}_{h_i}^0- {\bm v}_{h_i}\|\to 0$ as $i\to\infty$.
  Note that  by \eqref{2.32} we deduce
  \begin{align*}
    \|\mathrm{curl}H{\bm v}_{h_i}^0\|=\|\mathrm{curl}{\bm v}_{h_i}^c\|\lesssim \|{\bm v}_{h_i}\|_{h_i}.
  \end{align*}
This means that  $\{H{\bm v}_{h_i}^0\}$ is bounded in $\b H(\mathrm{curl},\Omega)$. Since $\b X$ is compactly imbedded into $\l2$, there is a subsequence of $\{H{\bm v}_{h_i}^0\}$ converging to some $\bm v_0$ in $\l2$. Hence   a subsequence of $\{{\bm v}_{h_i}\}$ will converge to  $\bm v_0$ in $\l2$ as well.
\end{proof}
The following uniform convergence can be derived from the discrete compactness property of $\b X_{h  }$.
\begin{theorem} 
There holds the uniform convergence
\begin{eqnarray*}
& \|{T}_{h  }-{T}\|_{\b L^2(\Omega)\to \b L^2(\Omega)}\rightarrow0,~h\rightarrow0,p\rightarrow\infty.
\end{eqnarray*}
\end{theorem}
\begin{proof}
Since   $\cup_{(h,p)\in\mathcal{H}}  U_{h  }$ are dense in   $  H^1_0(\Omega)$, respectively, we deduce from {Corollary 4.1 and (\ref{2.22sa})}: for any $\bm f\in \b L^2(\Omega)$
\begin{eqnarray}\label{2}
\|(  T  -  T_{h  }) {\bm f}\|_{}\le\|(  T  -  T_{h  }) {\bm f}\|_{h}\rt 0.
 \end{eqnarray}
 That is,  $  T_{h  }$ converges to $  T$ pointwisely in {$\b L^2(\Omega)$}. Thanks to  the discrete compactness
 of $\b X_{h  }$, $\cup_{(h,p)\in\mathcal{H}}T_{h  }B$ is a relatively compact set in $\b L^2(\Omega)$ where $B$ is the unit ball in $\b L^2(\Omega)$. In fact, Let us choose any sequence $\{\bm v_{h  }\}_{(h,p)\in\mathcal{H}}\subset B$. Note that $T$ is compact from $\b X$ to $\b L^2(\Omega)$ then $\{T\bm v_{h  }\}_{h\in\mathcal{H}}$ is a relatively compact set in $\b L^2(\Omega)$. Hence it holds the collectively compact convergence ${T}_{h  }\rightarrow {T}~in~\b  L^2(\Omega)\text{as}~h\rightarrow0,p\rightarrow \infty$. Noting $T,T_{h  }: \b L^2(\Omega) \rightarrow \b L^2(\Omega)$ are self-adjoint, due to Proposition 3.7 or Table 3.1
in \cite{chatelin}  the assertion is valid.
\end{proof}

\begin{remark}
The unform convergence in Theorem 4.2   is   valid on the mild polygonal $(d = 2)$ or polyhedral $(d = 3)$ meshes, provided that $\inf_{v\in U_h}|S\bm f- v|_1\to 0$ $(h\to0)$ in Corollary 4.1.  The fact $\inf_{v\in U_h}|S\bm f- v|_1\to 0$ cannot be guaranteed {when the edge(or face) number of polygonal (or polyhedral) elements is so large  that their nodes (or faces) is shared by few elements}.
\end{remark}

Using the  spectral approximation theory in \cite{babuska} we are in a position to give the estimate for IPDG eigenvalues.
\begin{theorem}
Let $\lambda_{h  }$ be an eigenvalue of (\ref{2.10}) converging to  the eigenvalue $\lambda$ of (\ref{2.9}) and $M(\lambda)\subset \{\bm v: A_K(\bm v)<\infty,\forall K\in\mathcal T_h\}$.  When $\{\mathcal T_{h}\}$ is a family of nonconforming meshes
\begin{eqnarray}\label{4.19}
|\lambda-{\lambda}_{h  }| \lesssim h^{2\min(s,p+1)-4}p^{7-2s},\\
\|\bm u-\bm u_{h  }\|_h \lesssim h^{\min(s,p+1)-2}p^{3.5-s};\label{l5}
\end{eqnarray}
when $\{\mathcal T_{h}\}$ is a family of conforming meshes
\begin{align}\label{4.19s}
|\lambda-{\lambda}_{h  }| &\lesssim h^{2\min(s,p+1)-4}p^{6-2s},\\
\|\bm u-\bm u_{h  }\| &\lesssim \varepsilon(h,p)h^{\min(s,p+1)-2}p^{3-s},\label{l3}\\
\|\bm u-\bm u_{h  }\|_h &\lesssim h^{\min(s,p+1)-2}p^{3-s};\label{l4}
\end{align}
where $s=\min_{K\in\mathcal T_h}s_K$  and ${M}(\lambda)$    denotes the space spanned by all eigenfunctions  corresponding to the eigenvalue $\lambda$.
\end{theorem}
\begin{proof} Let $\lambda_{h  }$ and $\lambda$ be the $mth$ eigenvalues of (\ref{2.10}) and (\ref{2.9}), respectively, and $\dim M(\lambda)=q$.
From   Theorem 7.2  (inequality (7.12)), Theorem 7.3 and Theorem 7.4
in \cite{babuska} we get
\begin{eqnarray}
&&|\lambda-{\lambda}_{h  }| \lesssim \sum\limits_{i,j=m}^{m+q-1}| ((T-T_{h  })\bm\varphi_{i},\bm\varphi_{j})|+\|(T-T_{h  }) |_{M(\lambda)} \|_{\b L^2(\Omega)\to \b L^2(\Omega)}^2,\label{l1}\\
&&\|\bm u_{h}-\bm u\|_{}\lesssim\|(T-T_{h  })\mid_{M(\lambda)}\|_{\b L^2(\Omega)\to \b L^2(\Omega)}^{},\label{l2}
\end{eqnarray}
where $\bm \varphi_{m},\cdots,\bm \varphi_{m+q-1}$ are a a set of  basis functions for $M(\lambda)$.
Note that  $S\bm f=S_h\bm f=0$ in (\ref{p2}) and (\ref{p3}) for $\bm f\in \b X$. According to   Theorem 4.1, the estimate \eqref{l3} is obtained from \eqref{l2}.
 Hence the following Garlerkin orthogonality holds:
 \begin{align*}
  ((T-T_{h  })\bm\varphi_{i},\bm \varphi_{j})=\wt a_h((T-T_{h  })\bm\varphi_{i},T\bm \varphi_{j})
  =\wt a_h((T-T_{h  })\bm\varphi_{i},(T-T_{h  })\bm \varphi_{j}).
 \end{align*}
Substituting it into \eqref{l1} , we deduce (\ref{4.19}) and  (\ref{4.19s}) from Theorems 3.2 and 3.3.
Note that $\bm u_{h}=\lambda_h T_{h}\bm u_{h}$ and $ \bm u=\lambda T \bm u$. By the boundedness of $T_{h}$ and Corollary 4.1 we derive
\begin{align*}
\|\bm u_{h}-\bm u\|_{h}&=\|\lambda_h T_{h}\bm u_{h}-\lambda T \bm u\|_{h}\nonumber\\
&\leq \| \lambda_h T_{h}\bm u_{h}-\lambda  T_{h}
\bm u\|_{h}+\| \lambda T_{h} \bm u- \lambda T
\bm u \|_{h}
\nonumber\\
&\lesssim \|\lambda_h\bm u_{h}-\lambda
\bm u\|_{} + \lambda\| T_{h} \bm u- T
\bm u \|_{h}\nonumber\\
&\lesssim |\lambda_h-\lambda|+
\|\bm u-\bm u_{h}\| + \inf_{\bm v\in\b V_h}\| T \bm u-
\bm v \|_{h}
\end{align*}
which together with \eqref{l3} and Theorem 3.3(or Theorem 3.2) yields the estimate \eqref{l4} (or \eqref{l5}).
\end{proof}

\section{Numerical experiment}

\indent In this section, the $h$-refinement on the mesh size and the $p$-refinement on {the polynomial degree will be adopted in the IPDG discretization of} the  quad-curl eigenvalue problem.
{Here we} use the data structure of FE mesh in  the package of iFEM \cite{chen} in the environment of MATLAB.

{We} consider the  eigenvalue problem   on
the  square $\Omega_S:=(-1,1)^2$, 
   the  L-shape {domain} $\Omega_L:=(-1,1)^2$ $\backslash \{(-1,0]\times [0,1)\}$, and the cube $(-1,1)^3$.


First {we solve} the eigenvalue problem {using $P_2$} polynomial space on triangular meshes. Meanwhile  the eigenvalue problem  is also computed on quadrilateral meshes {for}  comparative {purpose}.
The choice of $\eta_1$ and $\eta_2$ is not very sensitive to the computational accuracy on the triangular meshes.
{Let $\mathcal T_h$ be} a quadrilateral mesh{, 
and set $\eta_1=2.5$, $\eta_2=1.6$.} 
Then we compute the quad-curl eigenvalues on $\Omega_L$ and $\Omega_S$ using quadrilateral meshes (see the top of Figure 1 for the coarsest mesh). The same $\eta_1$ and  $\eta_2$ are used in the computation on the uniform triangular meshes.
The computational {results are listed in Tables 1-4}.
We  compute the convergence {rates}  for numerical eigenvalues using  the approximate formula
$\log(\frac{| \lambda_{h_i} -\lambda|}{| \lambda_{h_j} -\lambda|})/\log({h_i}/{h_j})$.
   With the exact eigenvalues unknown, we use {the numerical eigenvalues computed  by the curl-curl-conforming elements} in \cite{wang} on the finest  mesh as the reference values.

{From Tables 1-4, we see that the asymptotic convergence rates of $\lambda_{1,h}$, $\lambda_{2,h}$ on $\Omega_S$ and $\lambda_{2,h}$ on $\Omega_L$ are about 2}. However the convergence rate of $\lambda_{1,h}$ on $\Omega_L$ using triangular mesh is less than {2 while using quadrilateral meshes is about 2.
On the other hand, we observe from Tables 1-4 that numerical eigenvalues obtained from triangular meshes have higher accuracy than those obtained from} quadrilateral meshes.

{In addition, we perform} some numerical experiments on {non-uniform} triangular meshes with hanging nodes (see the bottom of Figure 1). We can use 7176 degrees of freedom (DOF)  to obtain approximation eigenvalues $\lambda_{1,h}\sim\lambda_{5,h}$ on the square {up to 2-3 digits}:   7.0E+02, 7.07E+02, 2.3E+03, 4.2E+03, 5.0E+03, and use 3816 DOFs  to obtain approximation eigenvalues $\lambda_{1,h}\sim\lambda_{5,h}$ on the  L-shape {domain up to 2 digits}: 33,
98,
3.8E+02,
4.0E+02,
6.8E+02. This indicates the robustness of the IPDG method for solving the quad-curl eigenvalues on nonconforming triangular meshes.

Next {we solve for the lowest 5 eigenvalues using $P_3$ polynomial space}. The associated numerical eigenvalues and their error curves are shown in Table 5 and the left side of Figure 2{, respectively. We see that} the convergence rate of this case is around $O(h^4)$ for all computed eigenvalues.

As {the last} numerical example in 2D, {we apply the IPDG  method to solve the eigenvalue problem} on the square using different polynomial degrees {with  fixed mesh size}. The right hand side of Figure 2 plots the error {curves in the semi-log chart with fixed  $h=1/8$ and  polynomial degrees $p$ ranging from 2 to 7}. It can be seen that the errors of the computed DG eigenvalues have a linear trend with respect to the local polynomial degrees in {the semi-log scale, which indicates exponential rate of convergence. Numerically,
it is}  $exp(-rp)$ with $r=2.5\sim3$.

{Finally, we show} a numerical example on the three dimensional cube $[-1,1]^3$ {with polynomial} degrees $p$ ranging from 5 to 9{. We divide the cube into six uniform tetrahedra and set $\eta_1=15.6$,} $\eta_2=1.35$. The computed  lowest five DG eigenvalues are listed in Table 6  and the associated error curves are plotted in Figure 3.  {Numerically,  the convergence rates} of the lowest five DG eigenvalues are $exp(-rp)$ with $r=0.10\sim0.15$.

\indent{\bf Acknowledgment.}
This work was supported {in  part} by the National Natural Science Foundation of China Grants{: 11871092, 12001130, 12131005, and NSAF 1930402},
China Postdoctoral Science Foundation no. 2020M680316, and Science and Technology Foundation of Guizhou Province
no. ZK[2021]012. 

We thank Baijun Zhang in CSRC for his extensive discussions on this topic.

\begin{figure}
  \centering
  \includegraphics[width=2.9in]{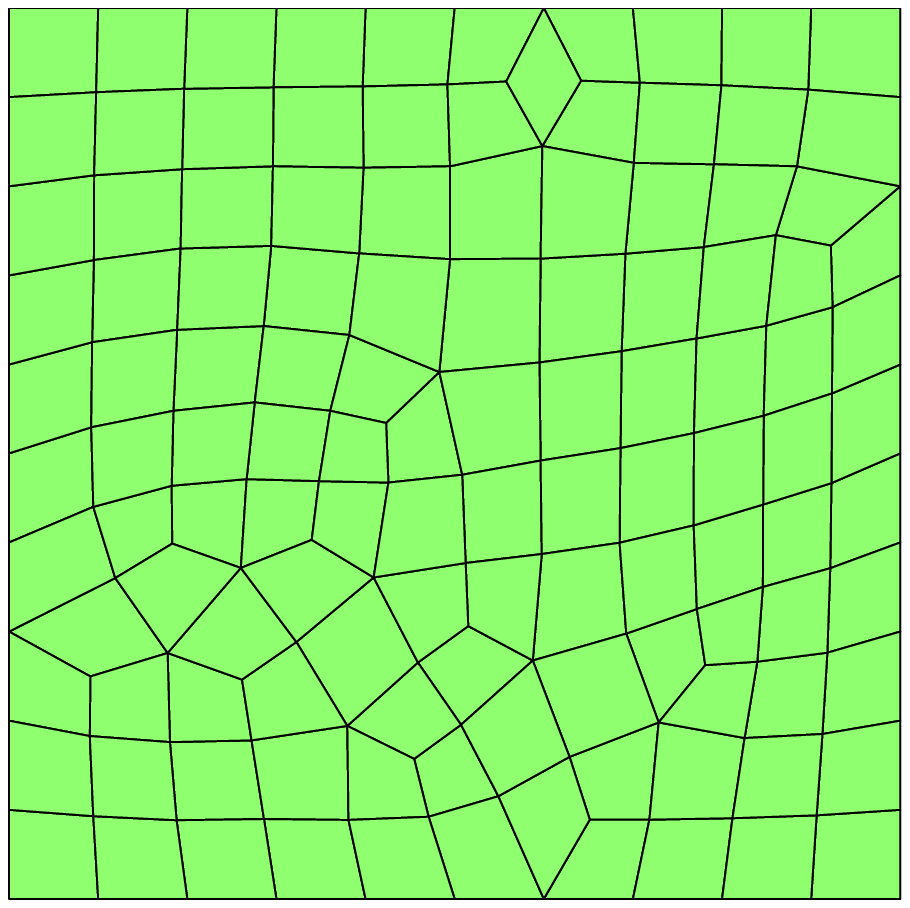}
  \includegraphics[width=2.9in]{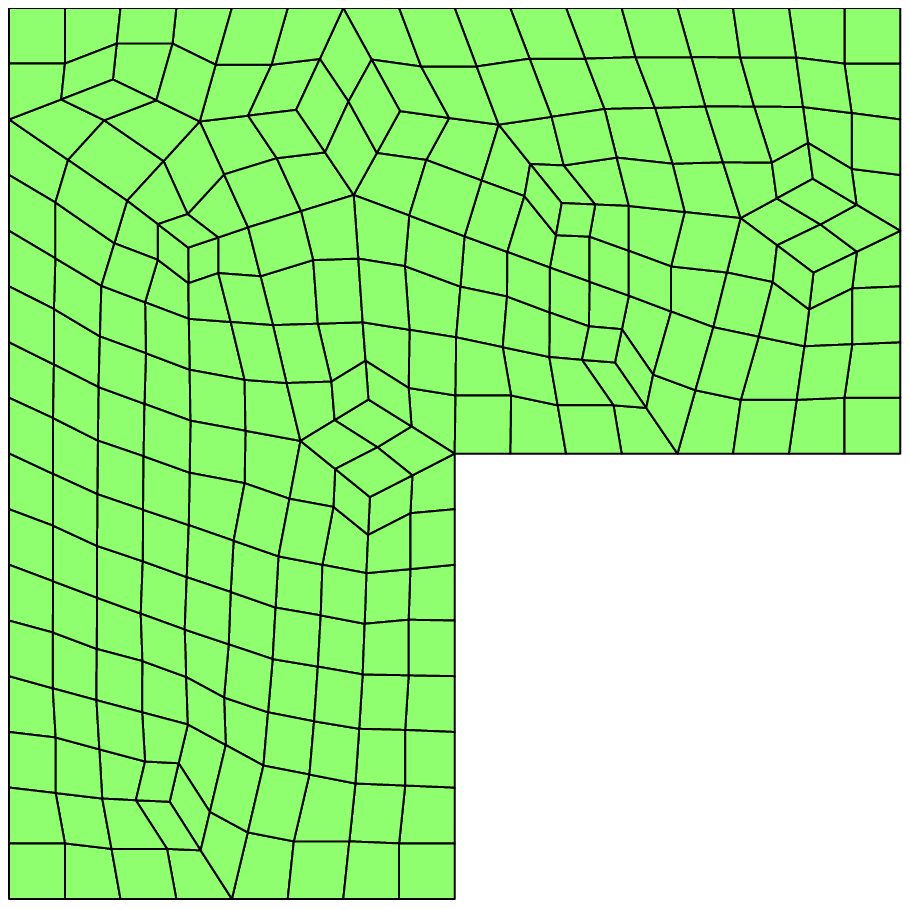}
 \includegraphics[width=2.9in]{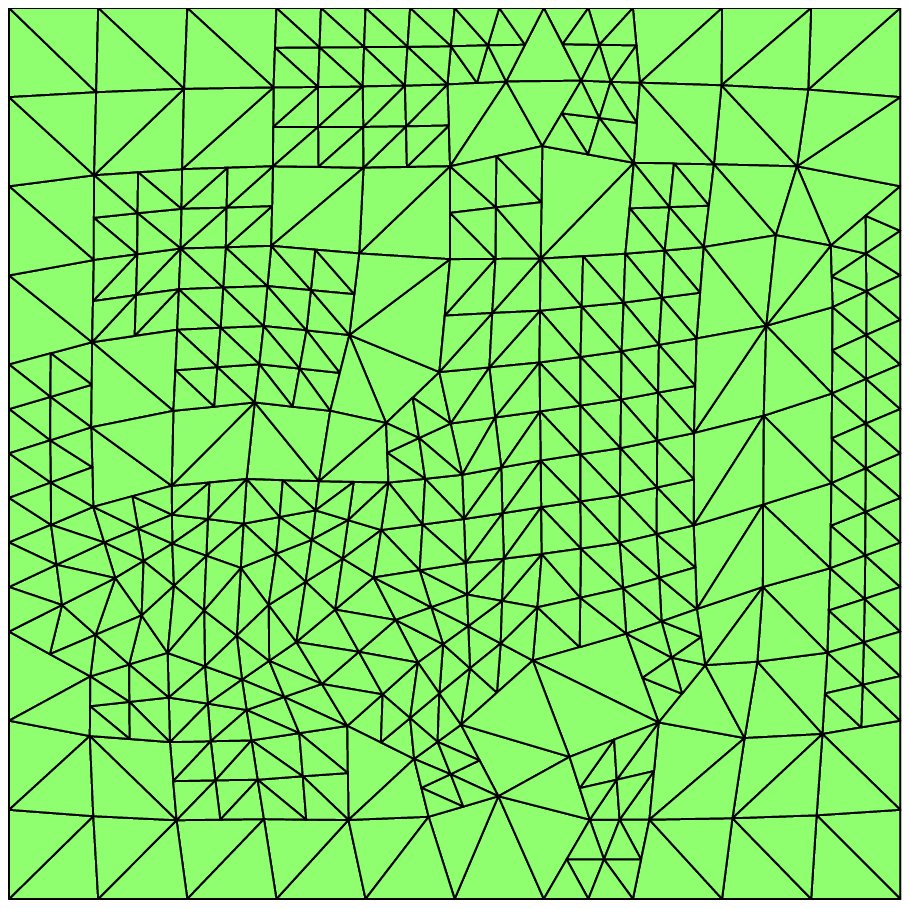}
  \includegraphics[width=2.9in]{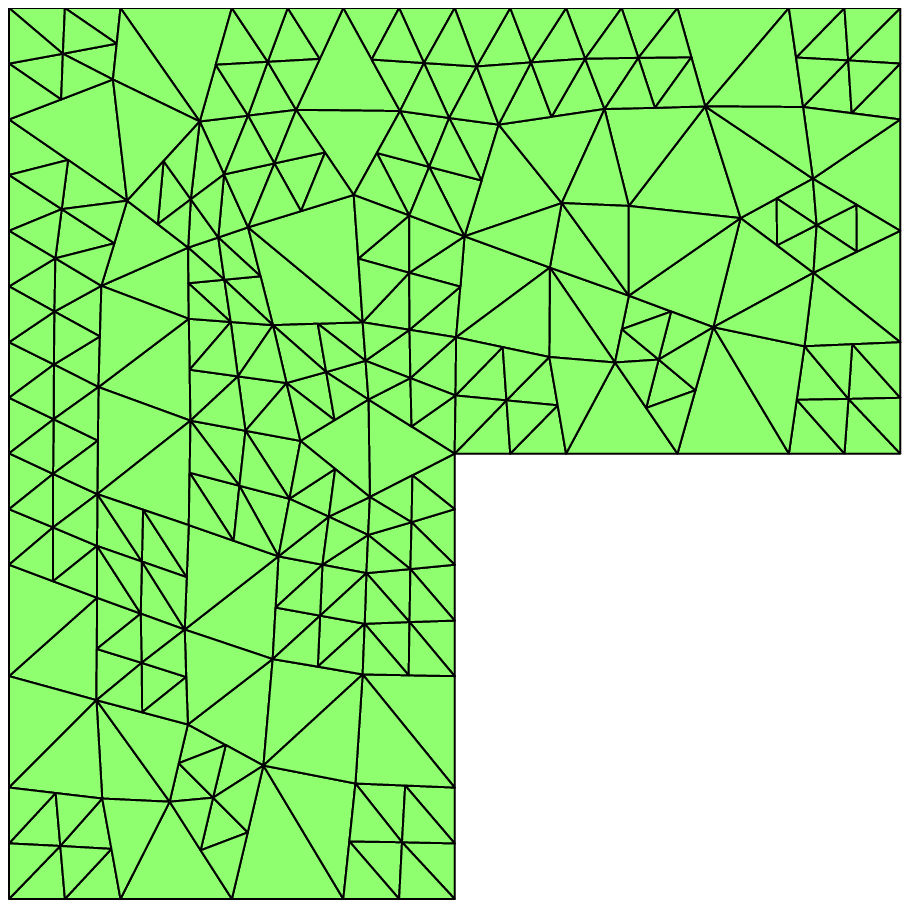}
  \caption{The quadrilateral meshes  on the square (left top, 1428 DOFs) and on the L-shape {domain}(right top, 3024 DOFs); 
   the triangular meshes with hanging nodes on the square (left bottom, 7176 DOFs, $\lambda_{1,h}\sim\lambda_{5,h}$:   7.0E+02, 7.07E+02, 2.3E+03, 4.2E+03, 5.0E+03) {and} on the L-shape {domain}(right bottom, 3816 DOFs, $\lambda_{1,h}\sim\lambda_{5,h}$: 33,
98,
3.8E+02,
4.0E+02,
6.8E+02)
}
\end{figure}


\begin{table}
\caption{Numerical eigenvalues on
 the square using quadrilateral meshes}
\begin{center} \footnotesize
\begin{tabular}{ccccccccccccccc}\hline
$h$&1/10&1/20&1/40&1/80&1/160\\\hline
$\lambda_{1,h}$& 762.9&	726.7&	713.4&	709.4&	708.3\\
$\lambda_{2,h}$&776.8&	730.7&	714.5&	709.7&	708.4\\
$\lambda_{3,h}$&2705.9&	2478.4&	2387.6&	2360.1&	2352.6\\
$\lambda_{4,h}$&3060.8&	4425.8&	4306.8&	4269.7&	4259.5\\
$\lambda_{5,h}$&4713.8&	5221&	5081&	5039.3&	5027.9\\\hline\hline
$|\lambda_{1,h}-\lambda_1|/\lambda_1$& 0.0776&    0.0265&    0.0077&    0.0020&    0.0005\\
Rate&&&& 1.9264&    2.1222\\
$|\lambda_{2,h}-\lambda_2|/\lambda_2$&0.0972&    0.0321&    0.0092&    0.0024&    0.0006\\
Rate&&&&1.9175&    2.0134\\
$|\lambda_{3,h}-\lambda_3|/\lambda_3$& 0.1515&    0.0546&    0.0160&    0.0043&    0.0011\\
Rate&&&&1.8949 &   1.9520\\
$|\lambda_{4,h}-\lambda_4|/\lambda_4$& 0.2808&    0.0399&    0.0120&    0.0033&    0.0009\\
 Rate&&&&1.8765&    1.9135\\
$|\lambda_{5,h}-\lambda_5|/\lambda_5$&0.0617&   0.0392&    0.0113&    0.0030&    0.0008\\
Rate&&&&1.8969&    1.9698\\
 \hline
\end{tabular}
\end{center}
\end{table}

\begin{table}
\caption{Numerical eigenvalues on
 the L-shape  using quadrilateral meshes}
\begin{center} \footnotesize
\begin{tabular}{ccccccccccccccc}\hline
$h$& 1/8&	1/16&	1/32&	1/64&	1/128\\\hline
$\lambda_{1,h}$&35.3209&	34.0426&	33.6188&	33.4995&	33.4685\\
$\lambda_{2,h}$&58.0695&	100.9415&	99.1287&	98.6054&	98.4652\\
$\lambda_{3,h}$&72.6584&	392.0341&	384.0611&	381.7741&	381.1653\\
$\lambda_{4,h}$&106.2517&	412.9257&	402.516&	399.245&	398.2771\\
$\lambda_{5,h}$&440.7255&	501.4575&	688.5538&	683.8663&	682.5761\\\hline\hline
$|\lambda_{1,h}-\lambda_1|/\lambda_1$& 0.0557&    0.0175&    0.0048&    0.0013&    0.0003\\
 Rate&&&&1.9386 &   1.9329\\
$|\lambda_{2,h}-\lambda_2|/\lambda_2$& 0.4098&    0.0256&    0.0072&    0.0019&    0.0005\\
Rate&&&&1.9135   & 1.9475\\
$|\lambda_{3,h}-\lambda_3|/\lambda_3$&0.8092&    0.0291&    0.0082&    0.0021&    0.0006\\
Rate&&&&1.9231    &1.9635\\
$|\lambda_{4,h}-\lambda_4|/\lambda_4$&0.7327&    0.0377&    0.0116&    0.0033&    0.0009\\
Rate&&&&1.7890    &1.8716\\
$|\lambda_{5,h}-\lambda_5|/\lambda_5$& 0.3538&    0.2648&    0.0094&    0.0026&    0.0007\\
Rate&&&&1.8817    &1.9369\\
 \hline
\end{tabular}
\end{center}
\end{table}

\begin{table}
\caption{Numerical eigenvalues on
 the square using uniform triangular meshes}
\begin{center} \footnotesize
\begin{tabular}{ccccccccccccccc}\hline
$h$&1/8&1/16&1/32&1/64&1/128\\\hline
$\lambda_{1,h}$&697.66& 	707.89& 	708.32& 	708.11& 	708.01\\
$\lambda_{2,h}$&703.97& 	709.47& 	708.67& 	708.19& 	708.03\\
$\lambda_{3,h}$&2294.94& 	2354.17& 	2353.56& 	2351.21& 	2350.34\\
$\lambda_{4,h}$&4112.82& 	4251.31& 	4259.09& 	4257.22& 	4256.24\\
$\lambda_{5,h}$&4912.41& 	5027.41& 	5028.64& 	5025.61& 	5024.45\\\hline\hline
$|\lambda_{1,h}-\lambda_1|/\lambda_1$&1.46e-02&   1.12e-04&   4.92e-04&   1.95e-04& 5.79e-05\\
Rate&&&&1.3344& 1.7510\\
$|\lambda_{2,h}-\lambda_2|/\lambda_2$&5.66e-03&   2.12e-03&   9.84e-04&   3.04e-04&8.33e-05\\
Rate&&&&1.6968&    1.8655\\
$|\lambda_{3,h}-\lambda_3|/\lambda_3$& 2.34e-02&   1.78e-03&   1.52e-03&   5.22e-04& 1.49e-04\\
Rate&&&& 1.5412&    1.8097\\
$|\lambda_{4,h}-\lambda_4|/\lambda_4$& 3.36e-02&   1.06e-03&   7.70e-04&   3.29e-04& 1.00e-04\\
Rate&&&&1.2244&    1.7186\\
$|\lambda_{5,h}-\lambda_5|/\lambda_5$&  2.22e-02&   6.80e-04&   9.25e-04&   3.21e-04&9.14e-05
\\
Rate&&&&1.5257&    1.8141\\
 \hline
\end{tabular}
\end{center}
\end{table}

\begin{table}
\caption{Numerical eigenvalues on
 the L-shape using uniform triangular meshes}
\begin{center} \footnotesize
\begin{tabular}{ccccccccccccccc}\hline
$h$&1/8&1/16&1/32&1/64&1/128\\\hline
$\lambda_{1,h}$&33.4608&	33.4824&	33.4664&	33.4603&	33.4589\\
$\lambda_{2,h}$&98.5124&	98.5541&	98.4667&	98.4312&	98.4204\\
$\lambda_{3,h}$&381.4047&	381.6117&	381.19&	381.0231&	380.9735\\
$\lambda_{4,h}$&396.4531&	398.4319&	398.177&	397.9822&	397.9045\\
$\lambda_{5,h}$&677.3528&	682.6124&	682.4738&	682.2339&	682.1436\\\hline\hline
$|\lambda_{1,h}-\lambda_1|/\lambda_1$&   9.86e-05&   7.44e-04&   2.66e-04&   8.37e-05&4.18e-05\\
Rate&&&&1.6684&    1.0000\\
$|\lambda_{2,h}-\lambda_2|/\lambda_2$& 9.72e-04&   1.40e-03&   5.08e-04&   1.47e-04&3.76e-05\\
Rate&&&&1.7859&    1.9705\\
$|\lambda_{3,h}-\lambda_3|/\lambda_3$&  1.18e-03&   1.72e-03&   6.16e-04&   1.78e-04&4.78e-05\\
Rate&&&&1.7915&    1.8973\\
$|\lambda_{4,h}-\lambda_4|/\lambda_4$&3.67e-03&   1.30e-03&   6.64e-04&   1.73e-04&2.26e-05\\
Rate&&&&1.9443&    2.9323\\
$|\lambda_{5,h}-\lambda_5|/\lambda_5$& 6.99e-03&   7.22e-04&   5.19e-04&   1.67e-04& 3.48e-05\\
Rate&&&&1.6343&    2.2661\\
 \hline
\end{tabular}
\end{center}
\end{table}

\begin{table}
\caption{Numerical eigenvalues on
 the square   {by different polynomial degrees $p$: uniform triangular mesh} $h=1/8$}
\begin{center} \footnotesize
\begin{tabular}{ccccccccccccccc}\hline
$p$&2&3&4&5&6&7\\\hline
$\lambda_{1,h}$&697.664507&	708.181644&	707.993329&	707.971329&	707.971765&	707.971509\\
$\lambda_{2,h}$&703.966943&	708.349258&	707.989007&	707.971929&	707.971702&	707.971551\\
$\lambda_{3,h}$&2294.939257&	2353.247385&	2350.206945&	2349.986846&	2349.987798&	2349.985790\\
$\lambda_{4,h}$&4112.822158&	4260.588288&	4256.264643&	4255.816486&	4255.817946&	4255.814058\\
$\lambda_{5,h}$&4912.406669&	5029.614780&	5024.259053&	5023.992937&	5023.992162&	5023.992341\\
 \hline
\end{tabular}
\end{center}
\end{table}

\begin{table}
\caption{Numerical eigenvalues on
 the cube {by different polynomial degrees $p$: 6 uniform tetrahedra} }
\begin{center} \footnotesize
\begin{tabular}{ccccccccccccccc}\hline
$p$&5&6&7&8&9\\\hline
$\lambda_{1,h}$&168.4810&	112.2701&	109.7608&	106.8450&	106.6736\\
$\lambda_{2,h}$&191.1329&	117.3074&	112.3088&	106.8450&	106.6834\\
$\lambda_{3,h}$&191.1329&	117.3074&	112.3088&	106.9811&	106.6737\\
$\lambda_{4,h}$&466.1273&	302.4162&	264.9300&	253.0533&	247.4130\\
$\lambda_{5,h}$&466.1273&	302.4162&	264.9300&	253.0533&	247.4129\\
 \hline
\end{tabular}
\end{center}
\end{table}

\begin{figure}[h]
  \centering
  \includegraphics[width=2.9in]{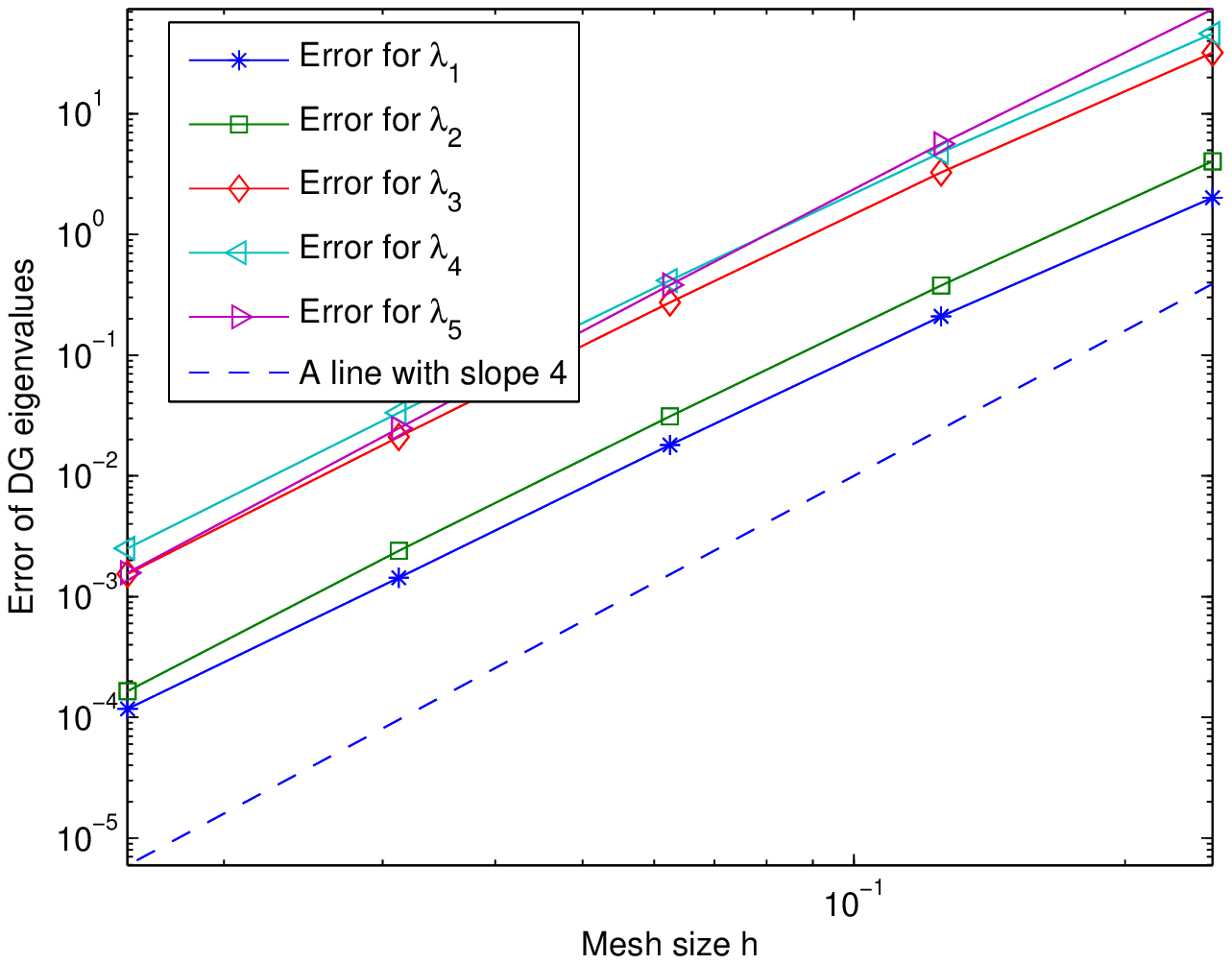}\includegraphics[width=2.9in]{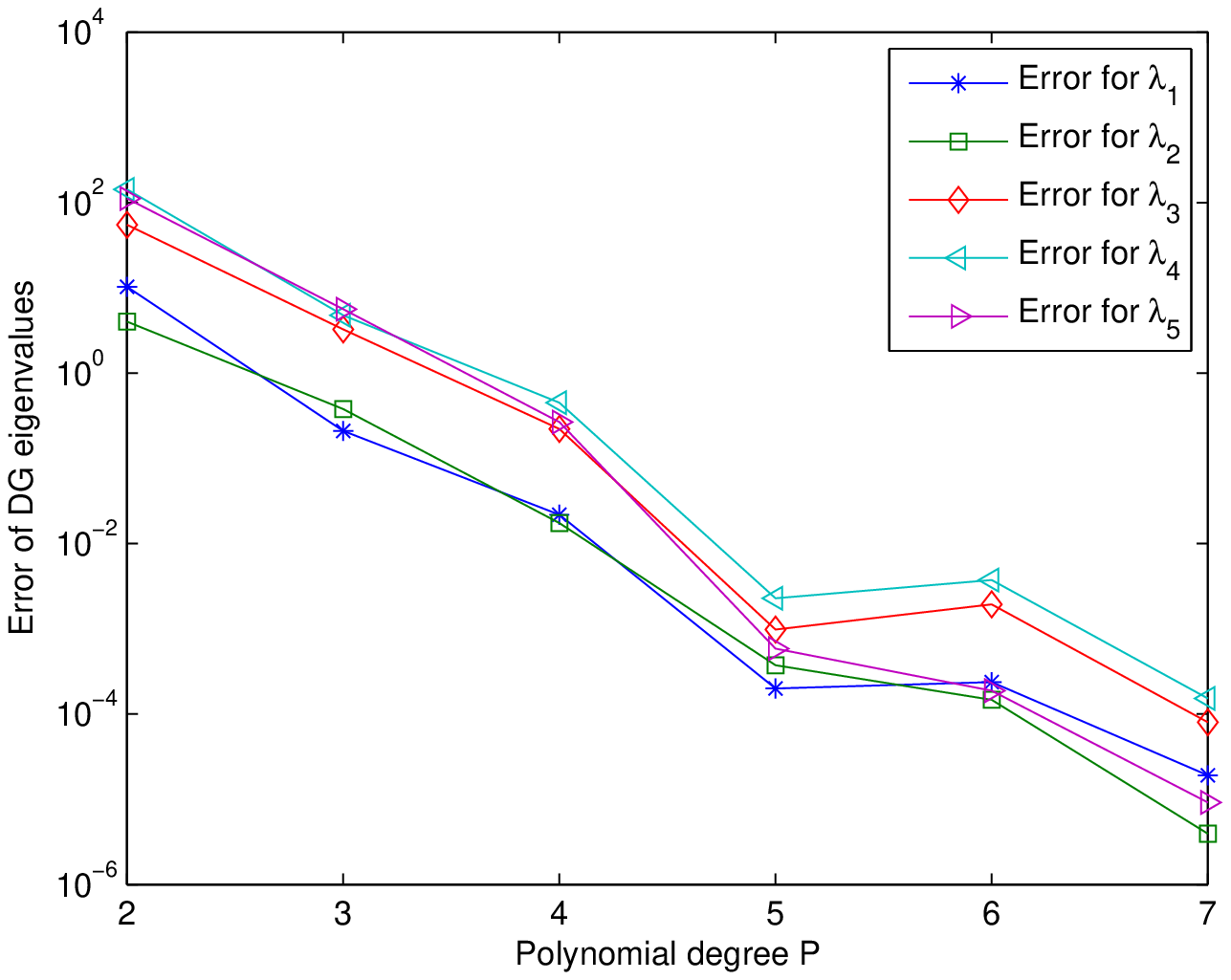}
    \caption{Error curves  on the square using {$P_3$ space with different mesh sizes (left) and different }polynomial degrees (right).}
\end{figure}
\begin{figure}
  \centering
  \includegraphics[width=2.9in]{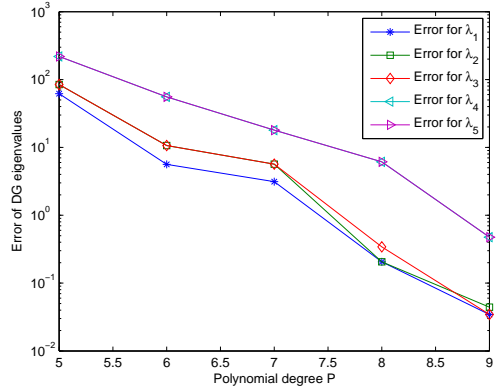}
    \caption{Error curves  on the cube using {different polynomial degrees with fixed mesh}. }
\end{figure}
\section*{Appendix A}
We introduce the following function (with the Legendre polynomial $L_i(\widehat x)$ of degree $i$)
 \begin{align}\label{a}
    \widehat\phi_i(\widehat x)=\left\{\begin{tabular}{ll}
                                      $\frac{1-\widehat{x}}{2}$,&$i=0$,\\
                                      $\frac{1+\widehat{x}}{2}$,&$i=1$,\\
                                      $\frac{L_i(\widehat x)-L_{i-2}(\widehat x)}{\sqrt{2(2i-1)}}$,&$i\ge2$.\\
                                   \end{tabular}\right.
 \end{align}
The basis functions of $\b H(\mathrm{curl})$-conforming rectangular element in $Q^{p-1,p}([-1,1]^2)\times Q^{p,p-1}([-1,1]^2)$ is as follows. 
\begin{description}
  \item[A.] Cell-based basis functions:
  \begin{description}
  \item[I.] $\bm\Phi^{-}_{i,j}:=\nabla\widehat \phi_i(\widehat x_1)\widehat \phi_j(\widehat x_2)-  \widehat   \phi_i(\widehat x_1)\nabla\widehat \phi_j(\widehat x_2)$ for $2\le i,j\le p$. \quad\textbf{($(p-1)^2$ in total)}
  \item[II.] $\bm\Phi^{+}_{i,j}:=\nabla\widehat \phi_i(\widehat x_1)\widehat \phi_j(\widehat x_2)+  \widehat   \phi_i(\widehat x_1)\nabla\widehat \phi_j(\widehat x_2)$ for $2\le i,j\le p$. \quad\textbf{($(p-1)^2$ in total)}
  \item[III.] $\bm\Phi_{1,j}:=\widehat \phi_j(\widehat x_2)\bm e_1$ and $\bm\Phi_{i,1}:=\widehat \phi_i(\widehat x_1)\bm e_2$ for $2\le i,j\le p$. \quad\textbf{($2(p-1)$ in total)}
  \end{description}
  \item[B.] Edge-based basis functions
  \begin{description}
  \item[I.] $\bm\Phi^+_{i,j}:=\nabla \widehat \phi_i(\widehat x_1)\widehat \phi_j(\widehat x_2)+ \widehat \phi_i(\widehat x_1)\nabla\widehat \phi_j(\widehat x_2)$ for $i=0,1$ and $2\le j\le p$.\quad\textbf{($2(p-1)$ in total)}
  \item[II.] $\bm\Phi^+_{i,j}:=\nabla \widehat \phi_i(\widehat x_1)\widehat \phi_j(\widehat x_2)+  \widehat \phi_i(\widehat x_1)\nabla\widehat \phi_j(\widehat x_2)$ for $j=0,1$ and $2\le i\le p$.\quad \textbf{($2(p-1)$ in total)}
  \item[III.] $\bm\Phi_{-1,j}:=\widehat \phi_j(\widehat x_2)\bm e_1$ and $\bm\Phi_{i,-1}:=\widehat \phi_i(\widehat x_1)\bm e_2$ for $i,j=0,1$.\quad\textbf{($4$ in total)}
  \end{description}
\end{description}
 The above basis functions are orthogonal in the sense of
$(\mathrm{curl}\cdot,\mathrm{curl}\cdot)_{}$ and $<\gamma\cdot,\gamma\cdot>$ where $\gamma \bm v:=\widehat{\bm v}\cdot \tau$ is the tangential trace of $\bm v$. 
Let $\widehat{\bm v}\in \nabla (Q^{1,p}+Q^{p,1})$.
Denote  $\Pi(\widehat{\bm v}):=(\Pi_1(\widehat{\bm v}),\Pi_2(\widehat{\bm v}))^T$ with
\begin{align}\label{5.2}
\begin{aligned}
\Pi_1(\widehat{\bm v})=\frac{1}{2}\int_{-1}^1\widehat{v}_1(\widehat x_1,-1)d\widehat x_1\phi_0(\widehat x_2)+\frac{1}{2}\int_{-1}^1\widehat{v}_1(\widehat x_1,1))d\widehat x_1\phi_1(\widehat x_2),\\
\Pi_2(\widehat{\bm v})=\frac{1}{2}\int_{-1}^1\widehat{v}_2(-1,\widehat x_2)d\widehat x_2\phi_0(\widehat x_1)+\frac{1}{2}\int_{-1}^1\widehat{v}_2(1,\widehat x_2))d\widehat x_2\phi_1(\widehat x_1).
\end{aligned}
\end{align}
Then $\widehat{\bm v}:=\bm r+\bm t+\Pi(\widehat{\bm v})$ with $\bm r= \Sigma_{i=0}^1\Sigma_{j=2}^p C^+_{i,j}\bm\Phi^+_{i,j}$ and $\bm t= \Sigma_{j=0}^1\Sigma_{i=2}^p C^+_{i,j}\bm\Phi^+_{i,j}$.
One can verify that
\begin{align}\label{5.3}
\begin{aligned}
\int_{[-1,1]^2}( r_1^2(1-\widehat x_2^2)^{-1}+r_2^2)d\bm x\lesssim \Sigma_{i=0}^1\Sigma_{j=2}^p (C^+_{i,j})^2((j^2-j)^{-1}+1),\\
\int_{[-1,1]^2} (t_2^2(1-\widehat x_1^2)^{-1}+t_1^2)d\bm x\lesssim \Sigma_{j=0}^1\Sigma_{i=2}^p (C^+_{i,j})^2((i^2-i)^{-1}+1).
\end{aligned}
\end{align}

\end{document}